\documentclass[twoside,pdf,11pt]{article}
\usepackage{amsfonts, amsbsy, amsmath, amsthm, amssymb, latexsym}
\usepackage{mathrsfs}
\usepackage{enumerate,centernot}
\usepackage[top=30mm,right=30mm,bottom=30mm,left=30mm]{geometry}
\usepackage[affil-it]{authblk}
\usepackage{latexsym}
\usepackage{longtable}
\usepackage{epsfig}
\usepackage{hhline}
\usepackage{amscd}
\usepackage{newlfont}
\usepackage{enumerate}
\usepackage{tabularx}
\usepackage{multirow}

\usepackage{bm}
\usepackage{fancyhdr}

\numberwithin{equation}{section}
\newcommand{\core}{\mathop{\mathrm{core}}}
\newcommand{\diag}{\mathop{\mathrm{diag}}}
\newcommand{\Stab}{\mathop{\mathrm{Stab}}}

\newcommand{\Spin}{\mathop{\mathrm{Spin}}}
\newcommand{\Sym}{\mathop{\mathrm{Sym}}}
\newcommand{\Aut}{\mathop{\mathrm{Aut}}}

\newcommand{\Ker}{\mathop{\mathrm{Ker}}}
\newcommand{\Alt}{\mathop{\mathrm{Alt}}}
\newcommand{\HS}{\mathop{\mathrm{HS}}}
\newcommand{\Fi}{\mathop{\mathrm{Fi}}}
\newcommand{\ML}{\mathop{\mathrm{McL}}}
\newcommand{\Co}{\mathop{\mathrm{Co}}}
\newcommand{\Out}{\mathop{\mathrm{Out}}}

\newcommand{\Soc}{\mathop{\mathrm{soc}}}

\newcommand{\Inndiag}{\mathop{\mathrm{Inndiag}}}
\newcommand{\Frat}{\mathop{\mathrm{Frat}}}
\newcommand{\charac}{{char}}
\newcommand{\pleq}{{\preccurlyeq}}

\newcommand{\wt}{\widetilde}

 \DeclareMathOperator{\Sp}{Sp}
 \DeclareMathOperator{\soc}{soc}
\DeclareMathOperator{\PDelta}{P\Delta}   \DeclareMathOperator{\POmega}{P\Omega}  
 
 \DeclareMathOperator{\M}{M}
\DeclareMathOperator{\Th}{Th}
\DeclareMathOperator{\BM}{BM}
\DeclareMathOperator{\J}{J} \DeclareMathOperator{\He}{He}
\DeclareMathOperator{\Vt}{\mathcal{V}}
\DeclareMathOperator{\Of}{{\bf O}}\DeclareMathOperator{\Sf}{{\bf S}}\DeclareMathOperator{\Lf}{{\bf L}}\DeclareMathOperator{\Uf}{{\bf U}}
\DeclareMathOperator{\Suz}{Suz}
\DeclareMathOperator{\HN}{HN}\DeclareMathOperator{\Ly}{Ly}

\DeclareMathOperator{\SL}{SL} 
  
\DeclareMathOperator{\GSp}{GSp}
\DeclareMathOperator{\SU}{SU}

\DeclareMathOperator{\PGL}{PGL} 
\DeclareMathOperator{\GU}{GU} 
\DeclareMathOperator{\GO}{GO}

\DeclareMathOperator{\Ru}{Ru}
\DeclareMathOperator{\GL}{GL} \DeclareMathOperator{\AGL}{AGL}

\DeclareMathOperator{\ON}{O'N}
\DeclareMathOperator{\End}{End} 
\DeclareMathOperator{\Gal}{Gal}

 \DeclareMathOperator{\PX}{PX}  \DeclareMathOperator{\PY}{PY}  \DeclareMathOperator{\PI}{PI}  \DeclareMathOperator{\PS}{PS} \DeclareMathOperator{\PGO}{PGO}
 \DeclareMathOperator{\X}{X}  \DeclareMathOperator{\Y}{Y} 
\DeclareMathOperator{\pgal}{P\Gamma L}

\newcommand{\ol}{\overline}

\renewcommand{\bf}{\textbf}

\newcommand{\leqs}{\leqslant}

\newtheorem{theorem}{Theorem}

\newtheorem{thm}{Theorem}[section]

\newtheorem{prop}[thm]{Proposition}
\newtheorem{lem}[thm]{Lemma}
\newtheorem{corol}[thm]{Corollary}

\theoremstyle{definition}
\newtheorem{defn}[thm]{Definition}
\newtheorem{rmk}[thm]{Remark}

\newtheorem{ex}[thm]{Example}

\author{Andrea Lucchini, Claude Marion, Gareth Tracey}
\date{\today}
\title{Generating maximal subgroups of finite almost simple groups}
\begin{document}
\maketitle

\noindent\textbf{Abstract.}  For a finite group $G$, let $d(G)$ denote the minimal number of elements required to generate $G$. In this paper, given a finite almost simple group $G$ and any maximal subgroup $H$ of $G$, we determine a precise upper bound for $d(H)$. In particular, we show that $d(H)\leq 5$, and that $d(H)\geq 4$ if and only if $H$ occurs in a known list. This improves a result of Burness, Liebeck and Shalev. The method involves the theory of crowns in finite groups.  

\section{Introduction}
For a finite group $G$, let $d(G)$ denote the minimal number of elements required to generate $G$. It is well-known that if $G$ is a nonabelian simple group, then $d(G)=2$. See \cite{AsGu,Miller,St}. More generally, if $G$ is almost simple with simple socle $T$ (that is, $T\leq G \leq {\Aut}(T)$ with $T$ a non-abelian simple group) then, by \cite{DL0}, $d(G)={\max}\{2,d(G/T)\}\leq 3$.

In this paper, we consider the corresponding result for maximal subgroups $H$ of almost simple groups $G$. In \cite{BLS}, Burness, Liebeck and Shalev prove that like $G$, $H$ can be generated by a bounded number of elements. More precisely, they show that $d(H)\leq 6$. Furthermore, they prove that $d(H)\leq 4$ if $G$ is simple (which is best possible), and suggest that this bound should also hold in the general case. We use the theory of crowns in finite groups to investigate their suggestion. We prove that $d(H)\le 5$, and that $d(H)\geq 4$ if and only if $H$ occurs in a known list (see Table 1). Before precisely stating our theorem, we require the following notation and terminology: For $1\le i\le 9$, we will write $\mathcal{C}_i$ for the Aschbacher class of maximal subgroups of an almost simple classical group (see \cite{Asch}, but note that we use the definitions of the classes $\mathcal{C}_i$ from \cite{KL}). The class $\mathcal{C}_9$ is sometimes written $\mathcal{S}$ in the literature. A finite group $G$ is said to be \emph{local} if $G$ has a nontrivial soluble normal subgroup.

Our main result reads as follows.
\begin{theorem}
Let $G$ be a finite almost simple group with socle $G_0$, and let $H$ be a maximal subgroup of $G$.  The following assertions hold.
\begin{enumerate}[(i)]
\item $d(H)\leq 5$.
\item $d(H)\geq 4$ if and only if one of the following holds:\begin{enumerate}[(a)]
\item The socle of $G$ is an alternating group of degree $n$; $G\in \{{\Alt}_n, {\Sym}_n\}$; $H=(T^k.({\Out}(T)\times {\Sym}_k))\cap G$ is of diagonal type (i.e. $n=|T|^{k-1}$  where  $T$ is non-abelian simple and $k>1$); ${\Sym}_k \leqs H$; and $d({\Aut}(T)\cap H)=3$. In this case, $d(H)=4$. Or
\item The socle of $G$ is of classical type, with field of definition $\mathbb{F}_q$ and natural module of dimension $n$, and $(G,H)$ is one of the pairs listed in Table 1.\end{enumerate}
Furthermore, $d(H)\geq 4$ if and only if $H$ has an elementary abelian factor group of order $2^{d(H)}$.
\end{enumerate}
\end{theorem}
\begin{table}[h!]
\begin{tabular}[t]{|c|p{3.5cm}|p{6.0cm}|p{3.5cm}|p{1.0cm}|}
  \hline
  \multicolumn{5}{ |c| }{Table 1}\\ 
  \hline
  Case & Conditions on $n$ and $q=p^f$ & Aschbacher class of $H$ & Extra conditions on $H$ & $d(H)$\\
  \hline\hline
  $\Lf$ & $n$ is even, $q$ is odd and $f$ is even & $\mathcal{C}_1$ & $H/H\cap G_0$ has an elementary abelian factor group of order $2^3$ & $4$\\
  \hline
 $\Of^{\pm}$ & $q$ is odd & $\mathcal{C}_1$ & $H/H\cap G_0$ has an elementary abelian factor group of order $2^3$ & $4$\\
  \hline
 $\Of^{\pm}$ & $q$ is odd, $f$ is even & $\mathcal{C}_1$ & $H/H\cap G_0$ has an elementary abelian factor group of order $2^3$ & $5$\\
  \hline
 $\Lf$ & $n$ is even, $q$ is odd and $f$ is even & $\mathcal{C}_2$ with $t>2$ subspaces in the associated $m$-decomposition & $H/H\cap G_0$ has an elementary abelian factor group of order $2^3$ & $4$\\
  \hline
   $\Of^{\pm}$ & $q$ is odd and $f$ is even & $\mathcal{C}_2$ with $t>2$ subspaces in the associated $m$-decomposition, and either $m\neq 1$ or $m\equiv\pm 1$ (mod $8$) & $H/H\cap G_0$ has an elementary abelian factor group of order $2^3$ or $D_8$ & $4$\\
  \hline
 $\Lf$ & $n$ is even, $q$ is odd and $f$ is even & $\mathcal{C}_4$, and the isometry groups of the subspaces of the associated tensor decompositions are both non-local & $H/H\cap G_0$ has an elementary abelian factor group of order $2^3$ & $4$\\
  \hline
   $\Of^{\pm}$ & $q$ is odd and $f$ is even & $\mathcal{C}_4$, and the isometry groups of the subspaces of the associated tensor decompositions are both non-local & $H/H\cap G_0$ has an elementary abelian factor group of order $2^2$ & $4$\\
  \hline
   $\Of^{\pm}$ & $q$ is odd and $f$ is even & $\mathcal{C}_4$, and the isometry groups of the subspaces of the associated tensor decompositions are both non-local & $H/H\cap G_0$ has an elementary abelian factor group of order $2^3$ & $5$\\
  \hline
 $\Lf$ & $n$ is even, $q$ is odd and $f$ is even & $\mathcal{C}_7$, and if the associated tensor decomposition has $t$ components, then either $t=2$ and $\frac{n}{t}\neq 2$ (mod $4$), or $\frac{n}{t}\neq 3$ (mod $4$) & $H/H\cap G_0$ has an elementary abelian factor group of order $2^3$ & $4$\\
  \hline
   $\Of^{\pm}$ & $q$ is odd and $f$ is even &  $\mathcal{C}_7$, and if the associated tensor decomposition has $t$ components, then either $t=2$ and $\frac{n}{t}\neq 2$ (mod $4$), or $\frac{n}{t}\neq 3$ (mod $4$)  & $H/H\cap G_0$ has an elementary abelian factor group of order $2^3$ & $4$\\
  \hline
\end{tabular}
\caption{The exceptional cases from the main theorem}\label{ta:excases}
\end{table}
\vspace{5mm}

Part (ii) of the theorem shows that there exist infinitely many almost simple groups possessing a maximal subgroup requiring $5$ generators. Thus, Part (i) is best possible, and shows that the suggestion in \cite{BLS} mentioned above fails to hold. We will also show that there are infinitely many pairs $(G,H)$ with the property that $G$ is a finite almost simple group; $H$ is a maximal subgroup of $G$; $d(H)=3$; and $H$ does not have an elementary abelian factor group of order $2^3$. Thus, Part (ii) is also best possible. See Example \ref{iEx}.
 
\vspace{5mm}

\noindent\textbf{Notation:} Throughout the paper we will, for the most part, use the notation from \cite{KL} for group names. In particular, $Z_n$ denotes a cyclic group of order $n$, although we do use $n$ instead when there is no ambiguity. Similarly, we will write $p^n$ in place of $(Z_p)^n$ to denote an elementary abelian group of order $p^n$, for a prime $p$. We will denote the alternating and symmetric groups by $\Alt_n$ and $\Sym_n$ respectively.  

The notation $\Frat(G)$ will denote the Frattini subgroup of the group $G$, while $\soc(G)$ denotes the socle of $G$. We will write $\Aut(G)$ and $\Out(G)$ for the automorphism and outer automorphism group of $G$, respectively.

The strategy for proving the theorem, and layout of the paper, is as follows: As mentioned above, our approach uses the theory of crowns in finite groups, which will be described in detail in Section \ref{Prelim}. We will then conclude the section with a restatement of our theorem in this language (see Theorem \ref{CrownLanguage}). In Section \ref{Tools}, we prove a series of lemmas comprising the main tools which we will use to prove the theorem. We then subdivide our proof according to the Classification of Finite Simple Groups, beginning in Sections \ref{s:sporadic} and \ref{Alternating} with the proof of the theorem in the cases when $\soc(G)$ is a sporadic or an alternating group. We then move on to the classical cases in Section \ref{Classical}, before completing the proof in Section \ref{ExSpor}, where we consider the almost simple groups with  exceptional socle. We conclude the paper with the example mentioned above, which shows that our result is best possible.   

\section{Crowns in finite groups and a restatement of the main theorem}\label{Prelim}
Let $G$ be a  nontrivial finite group. In this section, we recall several notions in the theory of \emph{crowns} in finite groups, inclusively those of a chief series for $G$, a $G$-group, equivalent $G$-groups and monolithic primitive groups. We use these to express $d(G)$ as a function of the chief factors of $G$.

\begin{defn}\label{chiefdef} A \emph{chief series} of a finite group $G$ is a normal series
$$1=N_0 < N_1 < \dots < N_n = G$$
of finite length with the property that for $i\in \{0,\dots, n-1\}$, $N_{i+1}/N_i$ is a minimal normal subgroup of $G/N_i$. The integer $n$ is called the \emph{length} of the series and the factors  $N_{i+1}/N_i$, where  $0\leq i \leq n-1$,   are called the \emph{chief factors} of the series.
\end{defn} 

A nontrivial finite group $G$ always possesses a chief series. Moreover, two chief series of $G$ have the same length, and any two chief series of $G$ are the same up to permutation and isomorphism. Thus, adopting the notation of Definition \ref{chiefdef}, we may define \emph{the chief length} of $G$ to be $n$, and \emph{the chief factors} of $G$ to be the groups $N_{i+1}/N_i$.

We can now begin our description of the theory of crowns in finite groups. First, we require some terminology. 
\begin{defn}\label{MonoDef} A finite group $L$ is called \emph{monolithic} if $L$ has a unique minimal normal subgroup $A$. If in addition $A$ is not contained in $\Frat(L)$, then $L$ is called a \emph{monolithic primitive group}.\end{defn}

Let $L$ be a monolithic primitive group and let $A$ be its unique minimal normal subgroup. For each positive integer $k$,
let $L^k$ be the $k$-fold direct product of $L$. The \emph{crown-based power of $L$ of size  $k$} is the subgroup $L_k$ of $L^k$ defined by
$$L_k=\{(l_1, \ldots , l_k) \in L^k  \mid l_1 \equiv \cdots \equiv l_k \ {\mbox{mod}\text{ } } A \}.$$
Equivalently, $L_k=A^k \diag L^k$. 

If a group $G$ acts on a group $A$ via automorphisms (that is, if there exists a homomorphism $G\rightarrow \Aut(A)$), then we say that $A$ is a \emph{$G$-group}. If $G$ does not stabilise any non-trivial subgroup of $A$, then $A$ is called an \emph{irreducible} $G$-group. Two $G$-groups $A$ and $B$ are said to be $G$-isomorphic, or $A\cong_G B$, if there exists a group isomorphism $\phi: A\rightarrow B$ such that 
$\phi(g(a))=g(\phi(a))$ for all $a\in A, g\in G$.  Following  \cite{JL}, we say that two  $G$-groups $A$ and $B$  are \emph{$G$-equivalent} and we put $A \sim_G B$, if there are isomorphisms $\phi: A\rightarrow B$ and $\Phi: A\rtimes G \rightarrow B \rtimes G$ such that the following diagram commutes:

\begin{equation*}
\begin{CD}
1@>>>A@>>>A\rtimes G@>>>G@>>>1\\
@. @VV{\phi}V @VV{\Phi}V @|\\
1@>>>B@>>>B\rtimes G@>>>G@>>>1.
\end{CD}
\end{equation*}

\

Note that two  $G$\nobreakdash-isomorphic $G$\nobreakdash-groups are $G$\nobreakdash-equivalent. In the particular case where $A$ and $B$ are abelian the converse is true: if $A$ and $B$ are abelian and $G$\nobreakdash-equivalent, then $A$ and $B$ are also $G$\nobreakdash-isomorphic.
It is proved (see for example \cite[Proposition 1.4]{JL}) that two  chief factors $A$ and $B$ of $G$ are $G$-equivalent if and only if either they are  $G$-isomorphic, or there exists a maximal subgroup $M$ of $G$ such that $G/\core_G(M)$ has two minimal normal subgroups $X$ and $Y$
$G$-isomorphic to $A$ and $B$ respectively. For example, the minimal normal subgroups of a crown-based power $L_k$ are all $L_k$-equivalent.

Recall that the Frattini group $\Frat(G)$ of a nontrivial finite group $G$ is nilpotent. The following terminology will be used frequently. 
\begin{defn} Let $G$ be a nontrivial finite group and let $H/K$ be a chief factor of $G$.
\begin{enumerate}[(i)]
\item We say that $H/K$ is \emph{Frattini} if $H/K\leqs \Frat(G/K)$. 
\item We say that $H/K$ is \emph{complemented} if there exists a subgroup $U$ of $G$ such that $UH=G$ and $U\cap H=K$. The group $U$ is said to be a \emph{complement} of $H/K$ in $G$. 
\end{enumerate}
\end{defn}

Since the Frattini subgroup of a finite group $G$ is nilpotent, and the only subgroup supplementing $\Frat(G)$ is $G$ itself, the following lemma is immediate.
\begin{lem}\label{l:1}
Let $G$ be a nontrivial finite group and let $A=H/K$ be a chief factor of $G$. 
\begin{enumerate}[(i)]
\item If $A$ is  abelian then $A$ is non-Frattini if and only if $A$ is complemented.  
\item If $A$ is nonabelian then $A$ is non-Frattini. 
\end{enumerate}
\end{lem}

For an irreducible $G$-group $A$, we define $\delta_G(A)$ to be the number of non-Frattini chief factors $G$-equivalent to $A$ in a chief series for $G$. Clearly the number $\delta_G(A)$ does not depend on the choice of chief series for $G$. Denote by $L_A$ the monolithic primitive group associated to $A$.
That is
$$L_{A}=
\begin{cases}
A\rtimes (G/C_G(A)) & \text{ if $A$ is abelian}, \\
G/C_G(A)& \text{ otherwise}.
\end{cases}
$$
If $A$ is a non-Frattini chief factor of $G$, then $L_A$ is a homomorphic image of $G$. More precisely, there exists a normal subgroup $N$ of $G$ such that $G/N \cong L_A$ and $\Soc(G/N)\sim_G A$. Consider now all the normal subgroups $N$ of $G$ with the property that $G/N \cong L_A$ and $\Soc(G/N)\sim_G A$:
the intersection $R_G(A)$ of all these subgroups has the property that  $G/R_G(A)$ is isomorphic to the crown-based power $(L_A)_{\delta_G(A)}$. The socle $I_G(A)/R_G(A)$ of $G/R_G(A)$ is called the \emph{$A$-crown} of $G$ and it is  a direct product of $\delta_G(A)$ minimal normal subgroups $G$-equivalent to $A$.

\begin{prop}\label{p:11}
Let $G$ be a nontrivial finite group and let $A=H/K$ be a non-Frattini chief factor of $G$. The following assertions hold.
\begin{enumerate}[(i)]
\item We have $I_G(A)=HC_G(A)$.
\item The group $I_G(A)/R_G(A)$ is the direct product of the $\delta_G(A)$ non-Frattini chief factors of $G$ that are $G$-equivalent to $A$. Moreover ${\soc}(G/R_G(A))=I_G(A)/R_G(A)$. 
\item The group $G/R_G(A)$ is isomorphic to the crown-based power $(L_A)_{\delta_G(A)}$ of $L_A$ of size $\delta_G(A)$. 
\end{enumerate}
\end{prop}

\begin{proof}
Part (i) is an easy exercise, while Part (ii) follows from \cite[Proposition 2.4]{Forster}. Finally, Part (iii) is \cite[Proposition 9]{DetLuc}. 
\end{proof}

We can now restate our main theorem in the language of crowns. In fact, as we will show in the sequel (see Corollary \ref{CLCor}), the following is stronger than the main theorem.
\begin{thm}\label{CrownLanguage} Let $G$ be an finite almost simple group with socle $G_0$, so that $G_0\le G\le \Aut(G_0)$. Fix a maximal subgroup $H$ of $G$, and a non-Frattini chief factor $A$ of $H$.\begin{enumerate}[(i)]
\item If $A$ is non-abelian, then $\delta_H(A)\le 2$.
\item If $A$ is abelian but non-central, then $\delta_H(A)\le 2$.
\item If $A$ is central, then $\delta_H(A)\le 3$, unless $|A|=2$ and $(G,H)$ is one of the pairs described in Part (ii) of the main theorem. In this latter case, $\delta_H(A)=4$ if $G$ is as in Part (a) of the main theorem. If $G$ is as in Part (b), then $\delta_H(A)=d$, where $d$ is as in the last column of Table 1.\end{enumerate}\end{thm}

\section{Tools from the theory of crowns}\label{Tools}
In this section we provide the key technical tools which will be used to prove Theorem \ref{CrownLanguage}. The first reads as follows.
\begin{lem}\label{l:12}
Let $G$ be a noncyclic finite group. The following assertions hold.
\begin{enumerate}[(i)]
\item There exists a monolithic primitive group $L$ and a positive integer  $k$ such that $L_k$ is an image of $G$ and  $d(G)=d(L_k)> d(L_{k-1})$. 
\item If  $L$ is a monolithic primitive group such that $L_k$ is an image of $G$ for some positive integer $k$ then there exists a non-Frattini chief factor $A$ of $G$ such that $L$ is isomorphic to the monolithic primitive group $L_A$ of $G$ associated to $A$ and  $k\leq \delta_G(A)$. Moreover if $d(L_k)=d(G)$ then  $d((L_A)_{\delta_G(A)})=d(G)$. 
\end{enumerate}
\end{lem}

\begin{proof}

We  consider Part (i). Let $N$ be a maximal normal subgroup of $G$ such that $d(G/N)=d(G)$ and $d(H)<d(G)$ for any proper quotient $H$ of $G/N$.  Set $K=G/N$. \\
Suppose first that $K$ has a unique minimal subgroup $M$. 
Note that $M$ is a chief factor of $K$ and ${\soc}(K)=M$.  By \cite[Theorem 1.1]{LucMen} $d(K)=2$ and $K/M$ is cyclic. Since $d(K)=d(K/\Frat(K))$ and  $d(H)<d(K)$ for any proper quotient $H$ of $K$, we must have $\Frat(K)=1$. In particular $K$ is a monolithic primitive group with $d(K)>d(K/{\soc}(K))$. Suppose now that $K$ has two distinct minimal subgroups. The argument used in the proof of \cite[Theorem 1.4]{DL} yields that there exists a monolithic primitive group $L$ and a positive integer $k$ such that $K\cong L_k$ and $d(L_k)>d(L_{k-1})$. This proves Part (i).
 
We now consider Part (ii). By assumption there is a normal subgroup $N$ of $G$, a monolithic primitive group $L$ and a positive integer $k$ such that $G/N\cong L_k$. Let $B={\soc}(L)$. Then $G$ has a chief factor $A$ isomorphic to $B$ and the monolithic primitive group  $L_A$ of $G$ associated to $A$ is isomorphic to $L$. Note that $L_k$ has $k$ non-Frattini chief factors $L_k$-equivalent to $B$. It follows that $G$ has at least $k$ non-Frattini chief factors $G$-equivalent to $A$. In particular $k\leq \delta_G(A)$. Suppose $d(G)=d(L_k)$. As $k\leq \delta_G(A)$, we have $d((L_A)_{\delta_G(A)})\geq d(G)$. Proposition \ref{p:11}(iii) now yields $d(G)=d((L_A)_{\delta_G(A)})$, as needed. 
 \end{proof}

We now state and prove the main tool which will be used to prove our theorem. It reads as follows.
\begin{prop}\label{p:13}
Let $G$ be a noncyclic finite group. For a non-Frattini chief factor $A$ of $G$, let $L_A$ be the monolithic primitive group of $G$ associated to $A$ and let $\delta_G(A)$ be the number of non-Frattini chief factors of $G$ which are $G$-equivalent to $ A$. If $A$  is abelian define $$r(A)={\dim}_{{\End}_{L_A/A}(A)}\ A,$$ $$s(A)={\dim}_{{\End}_{L_A/A}(A)}\ H^1(L_A/A,A),$$  $$\theta(A)=\left\{\begin{array}{ll} 0 & \textrm{if} \ A \ \textrm{is central}\\ 1 & \textrm{otherwise}, \end{array}\right.$$and $$h(A)=\theta(A)+\left \lceil \frac{\delta_G(A)+s(A)}{r(A)}\right \rceil$$  
The following assertions hold.
\begin{enumerate}[(i)]
\item We have $$d(G)={\max}_{A\ \textrm{non-Frattini}} \ d((L_A)_{\delta_G(A)})$$ where the maximum is taken over all non-Frattini chief factors $A$ of $G$. 
\item Suppose that for every non-abelian chief factor $A=S^n$ of $G$ we have $$\delta_G(A) \leq \frac{|A|}{2n|{\Out}(S)|}.$$ Then either $d(G)= 2$ or there is an abelian non-Frattini chief factor $B$ of $G$ such that $d(G)=d((L_B)_{\delta_G(B)})\geq 3$ and $d(L_B)>d(L_B/{\soc}(L_B))$.
\item  Under the assumption of (ii), if $d(G)>2$ then 
\begin{eqnarray*}
d(G) & = & 
{\max}_{\begin{array}{l}A \ \textrm{abelian}\\   A \ \textrm{non-Frattini} \end{array}} d((L_A)_{\delta_G(A)})\\
& = &  {\max}_{\begin{array}{l}A \ \textrm{abelian}\\   A \ \textrm{non-Frattini} \end{array}} h(A)\\
&\leq &   {\max}_{\begin{array}{l}A \ \textrm{abelian}\\   A \ \textrm{non-Frattini} \end{array}} \delta_G(A)+\theta(A)
\end{eqnarray*}
 where the maximum is taken over all abelian non-Frattini chief factors $A$ of $G$. Moreover if $d(G)=  d((L_A)_{\delta_G(A)})$, where $A\cong Z_2$ is a non-Frattini chief factor of $G$, then $d(G)=\delta_G(A)$.
 \item  Under the assumption of (ii), if $${\max}_{\begin{array}{l}A \ \textrm{abelian}\\   A \ \textrm{non-Frattini} \end{array}} h(A) \leq 3$$   where the maximum is taken over all abelian non-Frattini chief factors $A$ of $G$,
  then $d(G)\leq 3$. 
\end{enumerate}
\end{prop}

\begin{proof}
We first consider Part (i).  By Lemma \ref{l:12} there is monolithic primitive group $L$ such that $L_k$ is an image of $G$ and $d(G)=d(L_k)>d(L_{k-1})$.  Also there is a non-Frattini chief factor $A$ of $G$ isomorphic to ${\soc}(L)$ such that the monolithic primitive group $L_A$ associated to $G$ is isomorphic to $L$. Moreover, $d((L_A)_{\delta_G(A)})=d(G)$. In particular 
$$d(G) \leq {\max}_{A\ \textrm{non-Frattini}} \ d((L_A)_{\delta_G(A)}).$$ The result now follows from Proposition \ref{p:11}(iii). In the remaining part of the proof we will implicitly use the following consequence of Part (i): there is a non-Frattini chief factor $A$ of $G$ such that $d(G)=d((L_A)_{\delta_G(A)})$ and $d(L_A)>d(L_A/{\soc}(L_A))$.

We now consider Part (ii).  By Lemma \ref{l:12} there is a non-Frattini chief factor $C$ of $G$ such that  $d(G)=d((L_C)_{\delta_G(C)})>d(L_C/{\soc}(L_C))$.  Suppose for a contradiction that $d(G)>2$ and  there is no  non-Frattini abelian chief factor $B$ of $G$ such that  $d(G)=d((L_B)_{\delta_G(B)})\geq 3$ and $d(L_B)>d(L_B/{\soc}(L_B))$. Then by Part (i), $C$ must be nonabelian. Write $C=S^n$ where $S$ is a nonabelian simple group, set  $L=L_C$ and recall that ${\soc}(L)\cong C$.
 Since $L$ is a  monolithic primitive group with nonabelian socle, by \cite[Corollary 8]{DLM}, there is a function $\psi_L: \mathbb{N}\rightarrow \mathbb{N}$ such that for every $s\geq d(L)$, we have: 
\begin{equation}\label{e:21}
k\leq \psi_L(s) \quad \textrm{if and only if} \quad d(L_k) \leq s.
\end{equation}
Moreover by \cite[Proposition 10]{DLM}, there is an absolute constant $\gamma$ such that if $s\geq d(L)$  then 
\begin{equation}\label{e:22}
\psi_L(s)\geq \frac{\gamma |C|^{s-1}}{n|{\Out}(S)|} 
\end{equation}
and by \cite[Corollary 1.2]{LM} and \cite[Corollary 1.2]{DetLuc2}, $\gamma \geq 1/2$. 
We claim that $d(L) \leq 2$. Suppose not. Since $\delta_G(C)\leq  |C|/(2n|{\Out}(S)|)$, Equations (\ref{e:21}) and (\ref{e:22}) give $d(L_{\delta_G(C)})\leq d(L)$ and so  $d(L)=d(L_{\delta_G(C)})$. In particular
$d(L)>d(L/{\soc}(L))$ and $L$ is not cyclic. Since $L$ is a monolithic primitive group and $L$ is not cyclic, by \cite{LucMen}, $d(L)={\max}(2,d(L/{\soc}(L))$. It follows that $d(L)=2$, a contradiction.  Arguing by contradiction, we have established the claim, namely  $d(L)\leq 2$.
Since $d(L)\leq 2$ and $\delta_G(C)\leq  |C|/(2n|{\Out}(S)|)$, setting $s=2$ in Equations (\ref{e:21}) and (\ref{e:22}), we obtain  $d(L_{\delta_G(C)})\leq 2$ and so  $d(G)\leq 2$, a contradiction. This final contradiction establishes Part (ii).

We now consider Part (iii).  Since $d(G)>2$, by Part (ii), there is an abelian non-Frattini chief factor $B$ of $G$ such that $d(G)=d((L_B)_{\delta_G(B)})$ and $d(L_B)>d(L_B/{\soc}(L_B))$.  Since $d(L_B)>d(L_B/{\soc}(L_B))$, \cite[Proposition 6]{DLM} gives  $d((L_B)_{\delta_G(B)})=h(B)$. By \cite{AsGu} $s(B)<r(B)$ and so $h(B)\leq \theta(B)+\delta_G(B)$.   Finally, if $d(G)=  d((L_A)_{\delta_G(A)})$, where $A\cong Z_2$ is a non-Frattini chief factor of $G$, then $(L_A)_{\delta_G(A)}\cong Z_2^{\delta_G(A)}$ and so $d(G)=\delta_G(A)$. Part (iii) follows.

Finally, Part (iv) follows immediately from Parts (i), (ii) and (iii).
\end{proof}

Proposition \ref{p:13} (together with the observation that $|S|^n/(2n|\Out (S)|)\leq 2$ for every positive integer $n$ and every finite non-abelian simple group $S$) allows us to immediately deduce the main theorem from Theorem \ref{CrownLanguage}. 
\begin{corol}\label{CLCor} The main theorem follows from Theorem \ref{CrownLanguage}.\end{corol}

We conclude this section with four more lemmas concerning the possible shapes of chief factors in finite groups. The first three are elementary, but will be used frequently. Before we state them, we require a definition. 
 \begin{defn} Let $G$ be a finite group.\begin{enumerate}[(a)]
\item A \emph{subsection} of $G$ is a group $N/M$, where $N\le G$ and $M$ is normal in $N$. A subsection $N/M$ is called a \emph{section} of $G$ if $N$ and $M$ are both normal in $G$.
\item Let $N/M$ be a section of $G$, and let $A$ be a non-Frattini chief factor of $G$. Then $NR_G(A)/MR_G(A)$ is a normal subgroup of $G/R_G(A)$, so the intersection of $NR_G(A)/MR_G(A)$ with the socle $\soc(G/R_G(A))$ is isomorphic to $A^m$, for some $m\ge 0$. We define $\delta_{G,N/M}(A):=m$. That is, $\delta_{G,N/M}(A)$ is the number of non-Frattini chief factors of $G$ which are $G$-equivalent to $A$, and appear as a section of $N/M$.\end{enumerate}\end{defn}
\begin{rmk}\label{ObviousRemark} Let $G$ be a finite group, and let $N/M$ be a section of $G$. If $\delta_{G,N/M}(A)>0$ for some chief factor $A$ of $G$, that is if $A$ appears as a section of $N/M$, then we will write $A\pleq N$.\end{rmk}

We begin the series of lemmas mentioned above with a result concerning chief factors in groups with a cyclic normal subgroup: its proof is an easy consequence of the Jordan-H\"{o}lder Theorem.
\begin{lem}\label{l:sectioncyclic}
Let $G$ be a finite group having a cyclic normal subgroup $N$, say $N=Z_a$ for some $a\in \mathbb{N}$, and set $J=G/N$.  The following assertions hold.
\begin{enumerate}[(i)]
\item A chief factor of $G$ is either a section of $N$ or $J$. 
\item If $J$ has a cyclic normal subgroup $N/M$, say $N/M=Z_b$ for some $b\in \mathbb{N}$,   then a chief factor of $G$ is either a section of $M$, $N/M$ or $G/N$. 
\end{enumerate}
\end{lem}

The next result is a useful reduction lemma.
\begin{lem}\label{ElementaryLemma} Let $G$ be a finite group, and let $A$ be a non-Frattini chief factor of $G$. Suppose that $1=N_0\le N_1\le\hdots\le N_k=G$ is a normal series for $G$. Then \begin{enumerate}[(i)]
\item $\delta_G(A)=\sum_{i=1}^{k}\delta_{G,N_i/N_{i-1}}(A)$.
\item $\delta_{G,G/N_{i}}(A)=\delta_{G/N_{i}}(A)$.
\item If $N_i/N_{i-1}$ is cyclic, then $\delta_{G,N_i/N_{i-1}}(A)\le 1$.
\end{enumerate}\end{lem}
\begin{proof} If $N$ is normal in $G$ and $NR_G(A)/R_G(A)\cong V^m$, then the group $\soc(G/R_G(A))$ modulo $(NR_G(A)/R_G(A))$ is clearly isomorphic to $A^{\delta_G(V)-m}$. Hence, $\delta_G(A)={\delta_{G,N}}(A)+\delta_{G,G/N}(A)$. Part (i) now follows by an easy inductive argument.

Since $\Frat(G/N_i)=\Frat(G)N_i/N_i$, Part (ii) follows, and since the non-Frattini chief factors of a cyclic group are precisely the (cyclic) prime factors of its unique square-free quotient, Part (iii) follows.
\end{proof}

\begin{lem}\label{QuasisimpleLemma} Let $G$ be a finite perfect group. Then $Z(G)\le \Frat(G)$. In particular, this holds if $G$ is quasisimple.\end{lem}
\begin{proof} Suppose that $z$ is an element of $Z(G)$ of prime order, and that $z$ is not in $\Frat(G)$. Then there exists a maximal subgroup $H< G$, with $z\not\in H$. Hence, $G=H\langle z\rangle$. It then follows that $H\unlhd G$. This contradicts $G$ being perfect, since $G/H$ is abelian in this case. The result follows.\end{proof}

 We now investigate the case where certain subgroups of wreath products appear as sections in finite groups. This will be especially useful in our work on the $\mathcal{C}_2$ and $\mathcal{C}_7$ families in the classical cases. First, we need two definitions.
\begin{defn}\label{Full} Let $Q$ be a finite abelian group. The subgroup $K:=\{(x_1,\hdots,x_t)\text{ : }\prod_ix_i=1\}$ of $Q^t$ is called the \emph{fully deleted subgroup} of $Q^t$.\end{defn}

\begin{rmk} If the group $Q$ in Definition \ref{Full} is elementary abelian of order $p^a$ for a prime $p$, then $K$ is a module for the group $J:=\Sym_t$, via permutation of coordinates. It is called the \emph{fully deleted permutation module} for $J$ of dimension $a$ over the field $\mathbb{F}_p$.\end{rmk}

\begin{rmk}\label{anrem} Our next stated lemma requires a careful analysis of the chief factors of certain subgroups in a wreath product $E\wr J$, where $E$ is a finite group, and $J:=\Sym_t$. We will denote the base group of such a wreath product by $B=B(E\wr J)$. We will view $B$ as the direct product $B=E_1\times\hdots\times E_t$ of $t$ copies of $E$, and for a subgroup $L$ of $E$, we will write $L_i$ for the corresponding subgroup of $E_i$. Furthermore, we will write $B_L:=L_1\times\hdots\times L_t$. We will frequently use, and make no further mention, of these conventions.\end{rmk}

Before proceeding with the analysis mentioned in Remark \ref{anrem} above, it will be useful to introduce some further terminology. 
\begin{defn}\label{ExtraLargeDef} Let $E$ be a finite group, let $t\geq 2$ be a positive integer, and consider the wreath product $E\wr J$, where $J:=\Sym_t$. Let $B=B(E\wr J)$ be the base group. We call a subgroup $H$ of $E\wr J$ \emph{extra large} if all of the following hold.\begin{enumerate}[(a)]
\item $H\cap J\in\{\Alt_t,\Sym_t\}$. 
\item $H\cap B$ contains $B_F$ for some normal subgroup $F$ of $E$ such that\begin{enumerate}[(i)]
\item $E/F$ has abelian Frattini quotient $E/\Frat(E)F$;
\item $\Frat(E)\cap F=\Frat(F)$; 
\item $\delta_{E,F}(W)\le 1$ for all non-Frattini chief factors $W$ of $E$.\end{enumerate}
\item $(H\cap B)/(H\cap B_{\Frat(E)F})$ is the fully deleted subgroup of $(E/\Frat(E)F)^t$.\end{enumerate}
The normal subgroup $F$ of $E$ above will be called a \emph{source} of $H$.\end{defn}  

Our lemma can now be stated as follows.
\begin{lem}\label{StAnalysis} Let $H$ be an extra large subgroup of a wreath product $E\wr J$, where $E$ is a finite group and $J:=\Sym_t$, with $t\geq 2$. Let $B=B(E\wr J)$ be the base group, and let $F\unlhd E$ be a source of $H$. Assume that $G$ is a finite group with a normal series
$$1< H\le G.$$
such that $N_G(F_1)H=G$. Let $W$ be a non-Frattini chief factor of $E$. 
\begin{enumerate}[(1)]
\item Suppose that $W\pleq F$ is non-central, and that if $t=2$ then $E/C_E(W)$ is not an elementary abelian $2$-group. Then $B_W\cong W^t$ is a non-central chief factor of $G$ contained in $H\cap B$.
\item Suppose that either $W\pleq F$ is central, or that $t=2$ and $E/C_E(W)$ is an elementary abelian $2$-group. Let $A_{\diag}$ [respectively $A_{full}$] be the diagonal [resp. fully deleted] permutation modules for $J$ over the field $\mathbb{F}_p$, where $p=|W|$.
\begin{enumerate}[(a)]
\item Assume that $p\mid t$. Then $A_{diag}$, $A_{full}/A_{diag}$ and $B_W/A_{full}$ are chief factors of $G$ contained in $H\cap B$. Furthermore, $A_{diag}$ is Frattini, since it is not complemented in $B_W$.
\item Assume that $p\nmid t$. Then the $G$-module $B_W\cong W^t$ splits into a direct sum of two $G$-chief factors: the diagonal subgroup $A_{diag}$ of $B_W$, and the fully deleted permutation module $A_{full}$ in $B_W$.
\end{enumerate}
\item Suppose that $W\pleq E/F$. Then
\begin{enumerate}[(a)]
\item Assume that $p\mid t$. Then $A_{diag}$ and $A_{full}/A_{diag}$ are chief factors of $G$ contained in $H\cap B$. Furthermore, $A_{diag}$ is Frattini, since it is not complemented in $A_{full}$.
\item Assume that $p\nmid t$. Then the fully deleted permutation module $A_{full}$ is a chief factor of $G$ contained in $H\cap B$. In particular, we get one central $G$-chief factor; and one non-central $G$-chief factor of order $|W|^{t-1}$.
\end{enumerate}
\end{enumerate}
Finally, the group $H\cap B$ is normal in $G$, and the non-Frattini $G$-chief factors contained in $H\cap B$ are a subset of the groups listed in (1), (2) and (3) above.
\end{lem}
\begin{proof} We prove the claim by induction on $|E|$. Identify $E$ with $E_1$, and for any subgroup $K$ of $E$, recall that we write $K_i$ for the corresponding subgroup of $E_i$, and $B_K$ for the group $K_1\times K_2\hdots\times K_t\le B$. Let $\pi_i:B\rightarrow E_i$ be the projection maps. First note that the condition $N_G(F_1)H=G$ implies that $H$ acts transitively on the set $\Lambda$ of $G$-conjugates of $F_1$. Since $H\cap B$ is precisely the kernel of the action of $H$ on $\Lambda$, and $H\unlhd G$, it follows that $H\cap B\unlhd G$, as claimed. Note also that $H/H\cap B$ is isomorphic to either $\Alt_t$ or $\Sym_t$.

We now examine the chief factors of $G$ contained in $H\cap B$. Suppose first that $F$ is non-trivial, and let $W$ be a minimal normal subgroup of $E$ contained in $F$. The group $N_G(F_1)$ acts on $F_1$ via automorphisms, and the associated $N_G(F_1)$-conjugates of $W$ are normal subgroups of $F$ contained in $\soc(F)$: Let $X$ be the product of the distinct $N_G(F_1)$-conjugates of $W$ in $E$. Then $X$, being (equivalent to) a normal subgroup of $N_G(F_1)\geq E_1$, is a normal subgroup of $E$. Hence, the condition $N_G(F_1)H=G$ implies that $B_{X}$ is normal in $G$. 

Suppose first that $W\le \Frat(E)$. Then $W\le \Frat(E)\cap F=\Frat(F)$. Moreover, $\Frat(F_1)\ \charac \ F_1$, so $N_G(F_1)$ normalises $\Frat(F_1)$. Hence, $X\le \Frat(F)$ as well. Thus $B_{X}\le \Frat(F)^t\le \Frat(B_F)$. Since $B_F$ is subnormal in $G$, it follows that $B_{X}\le \Frat(G)$. Now, $H/B_{X}\le (E/X)\wr J$, and the series $1<H/B_{X}\le G/B_{X}$ satisfies the hypothesis of the lemma, with $E$ replaced by $E/X$ and $F$ replaced by $F/X$. The result then follows from the inductive hypothesis. 

So we may assume that $W$ is not contained in the Frattini subgroup of $E$. Then $X\cong W^m$, where $0\le m\le {\delta_{E,F}(W)}=1$. Hence, we must have $X=W$, since $X$ is non-trivial. In particular, it follows that $W_1$ is normalised by $N_G(F_1)$, and that $B_{W}=B_{X}$ is normal in $G$. We wish to examine the $G$-chief factors contained in $B_{W}$. We distinguish two cases:
\begin{enumerate}[(1)]
\item $W$ is non-abelian. Then $W\cong T^a$ for some non-abelian simple group $T$ and some positive integer $a$. We claim that $B_{W}$ is in fact a minimal normal subgroup of $G$ in this case. To prove this, suppose that $K$ is any non-trivial normal subgroup of $G$ contained in $B_{W}$. Then $K^{\pi_i}\unlhd (H\cap B)^{\pi_i}=E_i$, so $K^{\pi_i}$ is a normal subgroup of $E_i$ contained in $W_i$. Since $K$ is non-trivial and $H\cap J$ is transitive, it follows that $K^{\pi_i}=W_i$ for all $i$. Thus, $K\le B_{W}$ is a subdirect product of $T^{at}$. Since $K\unlhd B_{W}$ and $T$ is a non-abelian simple group, it follows that $K=B_{W}$, so $B_{W}$ is a minimal normal subgroup of $G$.
\item $W$ is abelian. Then $|W|=p^a$, for some prime $p$. For ease of notation, set $Y:=B_{W}$. Let $D\cong F$ be the diagonal subgroup of $B_F$, and consider the subgroup $I:=D\times (H\cap J)\le (H\cap B)(H\cap J)\le H$. By \cite[Proposition 2.4.5]{KL}, the $(H\cap J)$-module $Y\downarrow_{H\cap J}$ has an $(H\cap J)$-series
\begin{align}\label{Series0} 1=Y_0<Y_{1}<Y_1+Y_2\le Y_3=Y\end{align}
where $Y_1$ is the diagonal subgroup of $Y$, and splits into a direct sum of $a$ copies of the trivial $\mathbb{F}_p[H\cap J]$-module; and $Y_2$ splits into a direct sum of $a$ copies of the (irreducible) fully deleted permutation module for $H\cap J$ over $\mathbb{F}_p$. Also, $Y=Y_3=Y_1+Y_2$ if $p\nmid t$. If $p\mid t$, then $Y/(Y_1+Y_2)$ is a direct sum of $a$ copies of the trivial $\mathbb{F}_p[H\cap J]$-module. Next, $Y\downarrow_D$ splits as a direct sum 
$$Y=W_1\oplus W_2\oplus \hdots\oplus W_t.$$
(We caution the reader that these $W_i$, although completely reducible by Clifford's Theorem, are not necessarily irreducible as $D$-modules.) It follows that $Y\downarrow_I$ has an $I$-series 
\begin{align}\label{Series}1=A_0< A_1\le A_1+A_2\le A_3=Y\end{align}  
where $A_1\cong W_1\otimes 1$ has dimension $a$, and $A_2\cong Y_2$ has dimension $a(t-1)$. Furthermore, if $p\nmid t$, then $A_1\cap A_2$ is trivial and $Y=A_1\oplus A_2$. If $p\mid t$, then $A_1\le A_2$ and $Y/A_2$ has dimension $a$. 

Now, if $W$ is central then $a=1$ and, as an $H$-module, the series at (\ref{Series}) is equivalent to the series at (\ref{Series0}). Hence, it is in fact an irreducible $H$-series. This is because $H\cap B$ is in the kernel of the action of $H$ on $Y$ in this case, so we can just view $Y$ as a module for $H/H\cap J$, which is isomorphic to either $\Alt_t$ or $\Sym_t$. The irreducibility of the factors then follows from \cite[Proposition 2.4.5]{KL}, as above. Since $W_1$ and $B_{W_1}$ are both normalised by $N_G(F_1)$, the series (\ref{Series}) is also fixed by $N_G(F_1)$. Hence (\ref{Series}) is a $G$-series with irreducible factors, since $G=HN_G(F_1)$.

So we may assume that $W$ is non-central. If $W_1$ is not in $Z(F_1)$, then we may choose an element $x\in B_F$ with the property that $x^{\pi_1}$ does not centralise $A_1$, and $x^{\pi_i}=1$ for $i\geq 2$. If $W_1\le Z(F_1)$ and $t>2$, then choose $e\in E$ such that $e$ does not centralise $W$. Then we may choose an element $x\in H\cap B$ with $x^{\pi_1}\in F_1e$, $x^{\pi_2}\in F_2e^{-1}$, and $x^{\pi_i}\in F_i$ for $i> 2$. If $W_1\le Z(F_1)$, $t=2$, and there exists $e\in E$ such that $e^2$ acts non-trivially on $W$, then we may choose an element $x\in H\cap B$ with $x^{\pi_1}\in F_1e$ and $x^{\pi_2}\in F_2e^{-1}$. Then in each case, neither $A_1$ nor $A_2$ are $\langle x\rangle$-modules, so neither of them are $(H\cap B)$-submodules of $Y$. Thus, (\ref{Series}) implies that in either of these cases, $Y$ must be irreducible as an $(H\cap B)(H\cap J)$-module, and hence irreducible as a $G$-module.

Finally, if $W_1\le Z(F_1)$, $t=2$, and $E/C_E(W)$ is an elementary abelian $2$-group, then arguing as in the central case above, it is easy to see that (\ref{Series}) is in fact a $G$-series with irreducible factors. 
\end{enumerate}

Finally, we apply our inductive hypothesis: we have a normal series
$$1<H/B_{W}\le G/B_{W}$$
where $H/B_{W}\le (E/W)\wr J$, and $N_{G/B_{W}}(F_1/W_1)H/B_{W}\ge (N_G(F_1)B_{W}/B_{W}).(H/B_{W})=G/B_{W}$. Furthermore, (b) holds with $E$ replaced by $E/W$, and $F$ replaced by $F/W$. 

This gives us the chief factors of $G$ contained in $B_F$. To find the chief factors of $G$ contained in $H\cap B/B_F$, note that $H\cap B/B_F$ is the fully deleted subgroup of $(E/F)^t$. Let $W$ be a non-Frattini chief factor of $E/F$. Then $W$ is central, since $E/F$ is abelian. Let $A_{diag}$ [respectively $A_{full}$] be the diagonal [resp. fully deleted] permutation modules. Then, arguing as in the central case above we see that 
\begin{enumerate}[(a)]
\item If $p\mid t$, then $A_{diag}$ and $A_{full}/A_{diag}$ are the chief factors of $G$ contained in $H\cap B/B_F$. Furthermore, $A_{diag}$ is Frattini, since it is not complemented in $A_{full}$.
\item If $p\nmid t$, then the fully deleted permutation module $A_{full}$ is a chief factor of $G$ contained in $H\cap B$.
\end{enumerate}
Hence, the chief factors of $G$ contained in $H\cap B$ are a subset of those stated in (1)-(3) of the lemma.\end{proof}

The following will allow us to apply Lemma \ref{StAnalysis} in our proof of Theorem \ref{CrownLanguage} in the classical case.
\begin{lem}\label{StApplication} Let $E:=\PX_m(q)$ and $F:=\PY_m(q)$, where $(\X,\Y)$ runs through the symbols $\{(\GL,\SL),(\GU,\SU),(\GSp,\Sp),(\GO^{\epsilon},\Omega^{\epsilon})\}$. Then  
\begin{enumerate}[(i)]
\item $F\geq E'$;
\item $\Frat(E)\cap F=\Frat(F)$; 
\item All chief factors of $E$ contained in $F$ are non-central; and
\item $\delta_{E,F}(W)\le 1$ for all non-Frattini chief factors $W$ of $E$.
\end{enumerate}
\end{lem}

\begin{proof} In most cases the group $E$ is almost simple with socle $F$, and $E/F$ is abelian. In these cases therefore, the result is clear. If $m=1$ then $F=1$, and again the result is clear. The other possibilities for $\PY_m(q)$ are listed in \cite[Proposition 2.9.2]{KL}, and the result can be checked by direct computation in these cases.\end{proof}

We conclude this section by recording some important results concerning the outer automorphism groups of the nonabelian finite simple groups. Recall that a finite simple group $G_0$ of Lie type occurs as the derived subgroup of the fixed point group of a simple algebraic group $\mathbf{G}$ of adjoint type, defined over an algebraically closed field of prime characteristic $p$, under a Steinberg endomorphism $\sigma$, i.e. $G_0=(\mathbf{G}_{\sigma})'$. We use the standard notation $G_0=(\mathbf{G}_{\sigma})'=\mathbf{G}(q)$ where $q=p^f$ for some positive integer $f$. (We include the possibility that $G(q)$ is of twisted type). Also if $G_0$ is of orthogonal type in even dimension with associated non-degenerate quadratic form $Q$, we let $D(Q)$ be the discriminant of $Q$ (see \cite[\S2.5]{KL}). 
\begin{prop}
Let $G_0$ be a finite simple group. 
The outer automorphism group ${\Out}(G_0)$ of $G_0$ is given in Table \ref{ta:outgroup}. 
\end{prop}

\begin{table}[h!]
\begin{tabular}{lll}
\hline
$G_0$ & ${\Out}(G_0)$  & Remarks\\
\hline
${\Alt}_n$, $n\geq 5$ & $\left\{\begin{array}{l} Z_2 \\ Z_2\times Z_2 \end{array}\right.$ & $\begin{array}{l} \textrm{if}  \ n\neq 6 \\   \textrm{if} \  n=6\end{array}$  \\ 
\hline
${\Lf}_n(q)$ & $\left\{ \begin{array}{l} Z_{(n,q-1)}:Z_f:Z_2 \\ Z_{(2,q-1)}\times Z_f \end{array} \right.$ & $\begin{array}{l} \textrm{if}  \ n\geq 3 \\   \textrm{if} \  n=2\end{array}$ \\
\hline
${\Uf}_n(q)$ &  $Z_{(n,q+1)}:Z_{2f}$\\
\hline
${\Sf}_{2m}(q)$ & $\left\{\begin{array}{l} Z_2\times Z_f \\  Z_f \\ Z_{2f}   \end{array}\right.$  & $\begin{array}{l}  \textrm{if}\ q \ \textrm{odd}\\  \textrm{if} \ m\geq 3  \ \textrm{and}\  q \ \textrm{even} \\ \textrm{if} \ m=2   \ \textrm{and}\  q \ \textrm{even} \end{array}$\\
\hline
$\Of^{0}_{2m+1}(q)$, $q$ odd & $Z_2\times Z_f$ \\
\hline
${\Of}^{+}_{8}(q)$  & $\left\{\begin{array}{l}  {\Sym}_4\times Z_f \\ {\Sym}_3\times Z_f\end{array}\right.$& $\begin{array}{l}  \textrm{if} \ q \ \textrm{odd} \\ \textrm{if} \ q \  \textrm{even}\end{array}$  \\
\hline
${\Of}^{+}_{2m}(q)$, $m>4$ &  $\left\{\begin{array}{l}  D_8\times Z_f \\ Z_2\times Z_2\times Z_f \\ Z_2\times Z_f\end{array}\right.$& $\begin{array}{l}  \textrm{if} \ q \ \textrm{odd and} \ D(Q) \ \textrm{square}  \\ \textrm{if} \ q \  \textrm{odd and}\  D(Q) \ \textrm{non-square} \\ \textrm{if} \ q \ \textrm{even}\end{array}$  \\
\hline
${\Of}^{-}_{2m}(q)$ &  $\left\{\begin{array}{l}  D_8\times Z_f \\ Z_2\times Z_{2f} \\ Z_{2f}\end{array}\right.$& $\begin{array}{l}  \textrm{if} \ q \ \textrm{odd and} \ D(Q) \ \textrm{square}  \\ \textrm{if} \ q \  \textrm{odd and}\  D(Q) \ \textrm{non-square} \\ \textrm{if} \ q \ \textrm{even}\end{array}$  \\
\hline
$G_2(q)$ & $\left\{\begin{array}{l}  Z_f \\ Z_f:Z_2 \end{array}\right.$ &  $\begin{array}{l} \textrm{if} \ p\neq 3 \\ \textrm{if} \ p=3 \end{array}$ \\
\hline
$F_4(q)$ & $\left\{\begin{array}{l} Z_f\\ Z_f:Z_2 \end{array}\right.$ &  $\begin{array}{l} \textrm{if} \ p\neq 2\\ \textrm{if} \ p=2 \end{array}$\\
\hline
$E_6(q)$ & $Z_{(3,{q-1})}:Z_f:Z_2$& \\
\hline
$E_7(q)$ & $Z_{(2,q-1)}\times Z_f$ & \\
\hline
$E_8(q)$ & $Z_f$&  \\
\hline
${}^2B_2(q)$, $q=2^{2m+1}$ & $Z_f$& \\
\hline
${}^2G_2(q)$, $q=3^{2m+1}$ &$Z_f$& \\
\hline
${}^2F_4(q)$, $q=2^{2m+1}$ & $Z_f$& \\
\hline
${}^3D_4(q)$ & $Z_{3f}$&\\
\hline
${}^2E_6(q)$ & $Z_{(3,q+1):Z_{2f}}$&\\
\hline
$\M_{11}$, $\M_{23}$, $\M_{24}$, $\J_1$, $\J_4$, $\Ru$  $\Ly$,  & $1$& \\
$\Co_{1}$, $\Co_2$, $\Co_3$, $\Fi_{23}$, $\Th$, $\BM$, $\M$ & & \\
\hline
$M_{12}$,  $M_{22}$, $\J_2$, $\J_3$, $\HS$, $\Suz$, $\ML$,  & $Z_2$ & \\
$\He$, $\ON$, $\Fi_{22}$, $\Fi'_{24}$, $\HN$ & & \\
\hline
\end{tabular}
\caption{The outer automorphism group of a finite simple group}\label{ta:outgroup}
\end{table}

The next result follows from  \cite{DL0}. We include a different proof illustrating the method of crowns. 
\begin{prop}\label{p:alsimple}
Let $G$ be a finite almost simple group and let $G_0={\soc}(G)$. Then $d(G)\in \{2,3\}$. Moreover $d(G)=3$ if and only if  $G$ has a central non-Frattini chief factor $A\cong Z_2$ with $\delta_G(A)=3$. Moreover if $d(G)>2$ then $G_0={\Lf}_n(q)$ where $n\geq 4$ is even, $p$ is odd and $f$ is even or  $G_0={\Of}^{\epsilon}(q)$ where $\epsilon\in\{\pm1\}$, $p$ is odd and $f$ is even. 
\end{prop}

\begin{proof}
Since $G_0$ is the only non-abelian chief factor  of $G$, we have $\delta_G(G_0)=1$. In particular $2\delta_G(G_0)|{\Out}(G_0)|<|G_0|$ and so Proposition \ref{p:13} yields that $d(G)=2$ or

\begin{eqnarray*}
d(G) & = &   {\max}_{\begin{array}{l}A \ \textrm{abelian}\\   A \ \textrm{non-Frattini} \end{array}} d(L_{A,\delta_G(A)})\\
& = &  {\max}_{\begin{array}{l}A \ \textrm{abelian}\\   A \ \textrm{non-Frattini} \end{array}} \theta(A)+\left\lceil  \frac{\delta_G(A)+s(A)}{r(A)}\right \rceil \\
& \leq &  {\max}_{\begin{array}{l}A \ \textrm{abelian}\\   A \ \textrm{non-Frattini} \end{array}} \theta(A)  +\delta_G(A)
\end{eqnarray*}
where, in particular, $\theta(A)\in\{0,1\}$ is zero if and only if $A$ is a  central chief factor of $G$. 
Suppose that $d(G)>2$.
Let $A$ be a non-Frattini abelian chief factor of $G$.  Then $A$ can be viewed as a chief factor of $G/G_0 \leqs {\Out}(G_0)$. Set $K=G/G_0$.

Without loss of generality, suppose  $d(G)=d(L_{A,\delta_G(A)})$.  Note that ${\Out}(G_0)$ is not cyclic as otherwise $\delta_G(A)=1$, contradicting $d(G)>2$.

We claim that ${\Out}(G_0)$ is not the semidirect product of two nontrivial cyclic groups. Suppose otherwise. Say ${\Out}(G_0)=NH$ where $N\triangleleft {\Out}(G_0)$,   $H\cong {\Out}(G_0)/N$, and $N\cong Z_a$,  $H\cong Z_b$ for some positive integers $a$ and $b$ greater than 1.  Then $K$ has a cyclic normal subgroup, namely $N\cap K\leqs N \cong Z_a$, and 
$$K/(N\cap K)\cong (NK)/N \leq {\Out}(G_0)/N\cong Z_b.$$ Following Lemma \ref{l:sectioncyclic}, we can assume that $G_0={\Uf}_n(q)$ or $G_0={}^2E_6(q)$, as otherwise
either $\delta_G(A)=1$ or $A\cong Z_2$ is central and $\delta_G(A)=2$, contradicting $d(G)>2$.  
Using the relations satisfied by the two standard generators of ${\Out}(G_0)$, we obtain that either $\delta_G(A)=1$, or $\delta_G(A)=2$ and A is a central chief factor of $G$. In particular, $d(G) =2$, again, a contradiction. 

In the remainder of the proof we will use  Lemma \ref{l:sectioncyclic} implicitly. Suppose that ${\Out}(G_0)$ is the semidirect product of three nontrivial cyclic groups. We can assume that $G_0\neq E_6(q)$ as otherwise $\delta_G(A)=1$ or $A$ is a central chief factor of $G$ and $\delta_G(A)=2$, contradicting $d(G)>2$.  Hence $G_0={\Lf}_n(q)$ where $n\geq 3$, or $G_0={\Of}^+_{2m}(q)$ where $q$ is odd and $D(Q)$ is not a square.
Also if $G_0={\Lf}_n(q)$  so that ${\Out}(G_0)=Z_{(n,q-1)}:Z_f:Z_2$ then $p$ is odd, $n$ and $f$ are both even, and $A\cong Z_2$, as otherwise $\delta_G(A)=1$ or $A$ is a central chief factor of $G$ and $\delta_G(A)=2$, contradicting $d(G)>2$. 
 Similarly if $G_0={\Of}_{2m}^+(q)$ so that ${\Out}(G_0)=Z_2\times Z_2\times Z_f$, $f$ must be even and $A\cong Z_2$. The result follows in the case considered.
  
 Suppose that $G_0={\Of}^+_{8}(q)$. Note that $q$ is odd as otherwise ${\Out}(G_0)={\Sym}_3\times Z_f$ and either $\delta_G(A)=1$ or $A$ is central  and $\delta_G(A)=2$, contradicting $d(G)>2$. In particular, 
 ${\Out}(G_0)={\Sym}_4\times Z_f$.  Also $f$ is even and $A\cong Z_2$ as otherwise $d(G)=2$, a contradiction. The result follows in this case.
 
Suppose finally that $G_0={\Of}^\epsilon_{2m}(q)$, where $m>4$, $\epsilon\in \{\pm1\}$, $q$ is odd and $D(Q)$ is a square.  Then ${\Out}(G_0)=D_8\times Z_f$.  Again, $f$ must be even and $A\cong Z_2$ as otherwise $d(G)=2$, a contradiction. The result follows in this case. 
\end{proof}

\begin{lem}\label{l:index2sub2gen}
Let $G$ be a finite $m$-generated group and let $H$ be a subgroup of $G$ of index n. Then $d(H)\leq 1+n(m-1)$. 
\end{lem}

\begin{proof} 
This is a consequence of Schreier-Nielsen formula, see for example \cite[Proposition 12.1]{Lgeom}. 
\end{proof}

This completes our preparations. We will now prove Theorem \ref{CrownLanguage} by going through each of the relevant cases from the Classification of Finite Simple Groups.

\section{Almost simple groups with sporadic socle}\label{s:sporadic}

In this section we establish our main result for finite almost simple groups having as socle a sporadic group. 

\begin{prop}\label{p:soclespo}
Suppose that $G_0\leqs G \leqs {\Aut}(G_0)$ where $G_0$ is one of the 26  finite simple sporadic groups. If $H$ is a maximal subgroup of $G$ then $d(H)\leq 3$.  
\end{prop} 

\begin{proof}
If $G \notin \{{\Co}_1,\Fi_{22}.2,\Fi_{23},\Fi_{24}',\HN,\BM,\M\}$  then \cite{Gap}  yields that $d(H)\leq 3$ and in fact \cite{Gap} gives that $d(H)=2$ except possibly if 
$(G,H)$ is as follows: \\

$\begin{array}{ll}
(1) \quad (\HS,2\times {\Alt}_6.2^2) & \quad \quad (2) \quad (\HS.2, (2\times {\Alt}_6.2^2).2)\\ 
 (3) \quad (\Suz.2,3^{2+4}:2({\Sym}_4\times D_8)) & \quad \quad (4) \quad (\HN.2, ({\Sym}_6\times {\Sym}_6).2^2)\\
 (5)   \quad (\Fi_{22}, 3^{1+6}:2^{3+4}:3^2:2).&
\end{array}$\\

In cases (1), (2) and (4), MAGMA \cite{Magma} yields that $H/H'=Z_2\times Z_2 \times Z_2$ and so $d(H)=3$. In case (3), $H$ has a factor group $\Sym_4\times D_8$, which requires $3$ generators. Hence, $d(H)=3$. Finally in case (5) $H$ has a factor group isomorphic to $3^2:2$, where the $2$ acts by inverting the non-zero elements in the $3^2$. Thus, $d(H)=3$.

Suppose next that $G=\Co_1$. The maximal subgroups of $G$ are determined in \cite{WilsonCo1}. The group $G$ has 22 conjugacy classes of maximal subgroups. If $$H\in \{{\Co}_2,\ 3.{\Suz}:2,\ 2^{11}:{\M}_{24},\ {\Co}_3,\ 2^{1+8}\cdot O_8^{+}(2),\ U_6(2):{\Sym}_3\}$$ then \cite{Gap} yields that $d(H)=2$.  The other maximal subgroups $H$ of $G$ are (up to conjugacy): \\

$\begin{array}{lll}
(1)\quad ({\Alt}_4\times G_2(4)):2 &\quad \quad (2)\quad 2^{2+12}:(\Alt_8\times \Sym_3) & \quad \quad (3)\quad 2^{4+12}.(\Sym_3 \times 3.\Sym_6) \\
(4)\quad 3^2.U_4(3).D_8  &\quad \quad (5)\quad 3^6:2.\M_{12} & \quad \quad (6) \quad(\Alt_5\times \J_2):2 \\
(7) \quad  3^{1+4}:2.S_4(3).2& \quad \quad (8) \quad (\Alt_6\times U_3(3)).2 & \quad \quad (9) \quad 3^{3+4}:2.(\Sym_4\times \Sym_4) \\
(10) \quad \Alt_9\times \Sym_3& \quad \quad (11) \quad  (\Alt_7 \times L_2(7)):2  & \quad \quad (12) \quad D_{10}\times (\Alt_5\times \Alt_5.2).2\\
(13) \quad 5^{1+2}: {\GL}_2(5) & \quad \quad (14) \quad 5^3:(4\times \Alt_5).2 & \quad \quad (15) \quad  7^2:(3\times 2.\Sym_4)\\
(16) \quad  5^2:2.\Alt_5
\end{array}$

It follows that $d(H)\leq 3$.

Suppose that $G=\Fi_{22}.2$. The maximal subgroups of $G$ are determined in \cite{KleidmanFi22}. The group $G$ has  13 conjugacy classes of maximal subgroups $H$. Following \cite{WebAtlas} we have $d(H)=2$ except possibly if $H=2^{7}.S_6(2)$. 
However in the latter case, $H = M:2$, where $M$ is a maximal subgroup of $\Fi_{22}$ of shape $2^6:S_6(2)$, with $S_6(2)$ acting naturally and hence irreducibly on $2^6$. Therefore, a chief factor $A$ of $H$ is either $2^6$, $S_6(2)$ or $2$, with $\delta_H(A)=1$ in each case. It follows that $d(H)=2$.\\

Suppose that $G=\Fi_{23}$.  The maximal subgroups of $G$ are determined in \cite{KleidmanFi23}. The group $G$ has 14 conjugacy classes of maximal subgroups. If $$H \not \in \{3^{1+8}.2^{1+6}.3^{1+2}.2.{\Sym}_4,\ [3^{10}].(L_3(3)\times2),\ 2^{6+8}.({\Alt}_7\times {\Sym}_3),\ S_6(2)\times{\Sym}_4\}$$ then, by \cite{Gap}, $d(H)=2$.  If $H\in \{2^{6+8}.(\Alt_7\times \Sym_3),\ S_6(2)\times\Sym_4\}$, then $d(H)=2$ too. \\
Suppose that $H=3^{1+8}.2^{1+6}.3^{1+2}.2.{\Sym}_4$. The first three factors of $H$  in the latter series are extraspecial groups, and so, each  has as centre a group of prime order. In particular, if one of these groups of prime order is a chief factor of $H$ then it must be Frattini. It follows that $d(H)=2$. \\
Assume finally that $H=[3^{10}].(L_3(3)\times2)$. By \cite[Section 7]{KleidmanFi23}, $[3^{10}].L_3(3)$ is a $(2,3,13)$-group and so is 2-generated. It now follows from Proposition \ref {p:13} that  $H$ is 2-generated as well.

Suppose that $G=\Fi_{24}'$. The maximal subgroups of $G$ are determined in \cite{LintonWilson}.  The group $G$ has 22 classes (up to isomorphism) of maximal subgroups with representatives $H$ as follows:\\

$\begin{array}{lll}
(1)\quad  \Fi_{23} &\quad \quad (2)\quad 2 \cdot \Fi_{22}:2 & \quad \quad (3)\quad 2^{1+12}.3U_4(3).2  \\
(4)\quad 2^2\cdot U_6(2):\Sym_3 &\quad \quad (5)\quad (\Alt_4\times O_8^{+}(2):3):2 & \quad \quad (6)\quad 2^{3+12}.(L_3(2)\times \Alt_6)\\
 (7)\quad  2^{6+8}.(\Sym_3\times \Alt_8)&\quad \quad (8) \quad 2^{11}\cdot \M_{24} & \quad \quad (9)\quad (3\times O_8^{+}(3):3):2 \\
 (10)\quad 3^{1+10}:U_5(2):2 &\quad \quad (11)\quad  3^2.3^4.3^8.(\Alt_5\times 2\Alt_4).2& \quad \quad (12)\quad (3^2:2\times G_2(3))\cdot2\\
 (13)\quad  3^3.[3^{10}].{\GL}_3(3)&\quad \quad (14)\quad  3^7\cdot O_7(3)& \quad \quad (15)\quad 7:6\times \Alt_7 \\
 (16)\quad 29:14&\quad \quad (17)\quad O_{10}^{-}(2) & \quad \quad (18)\quad  \He:2\\
 (19)\quad (\Alt_5\times \Alt_9):2& \quad \quad (20)\quad U_3(3):2 & \quad \quad (21)\quad L_2(13):2 \\
 (22)\quad \Alt_6\times L_2(8):3\\
 \end{array}$
 
 It follows that if  $H \neq   3^3.[3^{10}].{\GL}_3(3)$ then $d(H)\leq 3$. If $H=3^3.[3^{10}].{\GL}_3(3)$ then, by \cite{LintonWilson}, $N_{\Fi_{24}}(H)=H.2$ and \cite{Gap} yields that $d(N_{\Fi_{24}}(H))=2$. It follows that $d(H)=2$. \\

 Suppose that $G=\HN$. The maximal subgroups of $G$ are determined in \cite{NortonHN}. The group $G$ has 14 conjugacy classes of maximal subgroups. If $$H\not \in \{({\Alt}_6\times {\Alt}_6).D_8,\ 2^{3+2+6}(3\times L_3(2)),\ 3^4:2.({\Alt}_4 \times {\Alt}_4).4\}$$ then \cite{Gap} yields $d(H)=2$. In the remaining cases we have $d(H)\leq 3$. \\
 
 Suppose that $G=\BM$. The maximal subgroups of $G$ are given in \cite{WilsonBM}. The group $G$ has 30 conjugacy classes of maximal subgroups. For half of the representatives  $H$ of these 30 classes, \cite{Gap} gives $d(H)=2$. The remaining representatives are: \\
  
  $\begin{array}{lll}
  (1)\quad 2^5.2^{10}.2^{20}.(\Sym_5\times L_3(2)) &\quad \quad (2)\quad   O_8^{+}(3):\Sym_4&\quad \quad (3)\quad (3^2:D_8 \times U_4(3).2^2).2\\
  (4) \quad 3^{2}.3^3.3^6.(\Sym_4\times 2\Sym_4)  &\quad \quad (5)\quad \Sym_5\times \M_{22}:2  & \quad \quad (6)\quad (\Sym_6\times L_3(4):2).2\\
  (7)\quad 5^3.L_3(5)  &\quad \quad (8)\quad (\Sym_6\times \Sym_6).4  & \quad \quad (9)\quad 5^2:4\Sym_4\times \Sym_5\\
  (10)\quad L_2(49).2  &\quad \quad (11)\quad  L_2(31) & \quad \quad (12)\quad \M_{11}\\
  (13)\quad  L_3(3)&\quad \quad (14)\quad L_2(17):2  & \quad \quad (15)\quad L_2(11):2\\
\end{array}$

It  follows that  $d(H)\leq 3$ except possibly if $H$ is  as in case (3). In this case $H=N_G(E)$, where $E$ is elementary abelian of order $3^2$. This is shown in \cite{WilsonInvent}, where it is also proved that $C_G(E)3^2 \times U_4(3).2^2\le H$, and that the $2$ on top of $H$ induces a diagonal automorphism of $U_4(3)$. Since $\Out(U_4(3))\cong D_8$, and $H/C_G(E)\le \GL_2(3)$ has Sylow $2$-subgroup the semidihedral group $SD_{16}$, it follows that $H/(3^2\times U_4(3))$ is a subdirect product of $SD_{16}\times D_8$ of index $2$. Hence, $\delta_H(Z_2)=3$. Since it is clear that $\delta_H(E)$, $\delta_H(U_4(3))\le 1$, we have $d(H)=3$.   

Suppose finally that $G=\M$. At this stage, 44 conjugacy classes of maximal subgroups $H$ of $G$ have been identified, namely the 43 in \cite{WebAtlas} and the one with representative $L_2(41)$ (see \cite{NortonM}).
Moreover any further, if any,  maximal subgroup of $G$ is almost simple with socle $L_2(13)$, $U_3(4)$, $U_3(8)$ or ${}^2B_2(8)$. 
As seen, if $H$ is almost simple then $d(H)\leq 3$. Also \cite{WebAtlas} yields that 37 of the 43 representatives $H$ in \cite{WebAtlas}  satisfy $d(H)=2$. The remaining representatives are:\\

$\begin{array}{ll}
(1)\quad 2.\BM &\quad \quad (2)\quad 2^{1+24}.\Co_1\\
 (3)\quad 2^{10+16}.\Omega_{10}^+(2) &\quad \quad (4)\quad 3^{1+12}.2.\Suz:2\\
 (5)\quad  2^{5+10+20}.(\Sym_3\times L_5(2)) & \quad \quad (6) \quad  2^{3+6+12+18}.(L_3(2)\times 3 \Sym_6)  \\
\end{array}$

Clearly, $d(H)=2$ in these remaining cases.

\end{proof}

\section{Almost simple groups with alternating socle}\label{Alternating}
In this section we prove the following. 

\begin{prop}\label{p:soclealt}
 Suppose that $G_0\leqs G \leqs {\Aut}(G_0)$ where $G_0={\Alt}_n$ is an alternating group of degree $n\geq 5$. If $H$ is a maximal subgroup of $G$ then $d(H)\leq 4$. Moreover $d(H)=4$ if and only if  $G\in \{{\Alt}_n, {\Sym}_n\}$, $H=(T^k.({\Out}(T)\times {\Sym}_k))\cap G$ is of diagonal type (i.e. $n=|T|^{k-1}$  where  $T$ is non-abelian simple and $k>1$), ${\Sym}_k \leqs H$ and $d({\Out}(T)\cap H/T^k)=3$. 
\end{prop}

Recall that if $n\neq 6$ then $G={\Alt}_n$ or $G={\Sym}_n$, whereas if $n=6$ then $G\in \{{\Alt}_6, {\Sym}_6, M_{10}, {\PGL}_2(9), {\pgal}_2(9)\}$.

We first treat the case where $G$ is neither an alternating group nor a symmetric group. In particular, $n=6$ and $G\in\{M_{10}, {\PGL}_2(9), {\pgal}_2(9)\}.$ We use MAGMA to check our calculations. 
\begin{lem}
Suppose that $G=M_{10}$ and let $H$ be a maximal subgroup of $G$. Then $d(H)=2$.
\end{lem}
\begin{proof}
By \cite{ConLee} $H\in\{Q_8:2, D_{10}.2, M_9,{\Alt}_6\}$.  Clearly, $d({\Alt}_6)=2$.  Suppose now that $H=D_{10}.2$. Since the chief factors of $H$ are $Z_5$, $Z_2$, $Z_2$, it follows that $d(H)=2$.
Suppose next that $H=Q_8:2$. Since the chief factors of $H$ are $Z_2$, $Z_2$, $Z_2$, $Z_2$ and $|\Frat(H)|=4$, we must have $d(H)=2$.
Suppose finally that $H=M_9=3^2:Q_8$.  Since the chief factors of $H$ are $Z_3\times Z_3$, $Z_2$, $Z_2$, $Z_2$  and $\Frat(H/3^2)=Z_2$, it follows that $d(H)=2$. 
\end{proof}

\begin{lem}
Suppose that $G={\PGL}_2(9)$ and let $H$ be a maximal subgroup of $G$. Then $d(H)=2$.
\end{lem}

\begin{proof}
By \cite{ConLee} $H\in\{D_{16}, D_{20}, 3^2.8,{\Alt}_6\}$.  Clearly,  if $H$ is an alternating group or a dihedral group then $d(H)=2$.  Suppose finally that $H=3^2.8$. Since the chief factors of $H$ are $Z_3\times Z_3$, $Z_2$, $Z_2$, $Z_2$  and $\Frat(H/3^2)=Z_4$, it follows that $d(H)=2$.
\end{proof}

\begin{lem}
Suppose that $G={\pgal}_{2}(9)$ and let $H$ be a maximal subgroup of $G$. Then $d(H)\leq 3$.
\end{lem}

\begin{proof}
By \cite{ConLee} $H\in\{Q_8:2^2, (D_{10}.2)\times 2, N_G(3^2),{\Sym}_6, M_{10}, {\PGL}_2(9)\}$. \\If $H\in \{{\Sym}_6, M_{10}, {\PGL}_2(9)\}$ then $H={\Alt}_6.2$ and it follows that $d(H)=2$.

Next, suppose that $H=Q_8:2^2$. Since the chief factors of $H$ are $Z_2$, $Z_2$, $Z_2$, $Z_2$, $Z_2$ and $\Frat(H)=Z_4$, it follows that $d(H)\leq 3$ and in fact $d(H)=3$.
Suppose that $H=N_G(3^2)$. Then $|H|=3^2\cdot2^4$. Since the chief factors of $H$ are $Z_3\times Z_3$, $Z_2$, $Z_2$, $Z_2$, $Z_2$  and $\Frat(H/3^2)=Z_4$, it follows that $d(H)=2$. Suppose finally that $H=(D_{10}.2)\times 2$. Since the chief factors of $H$ are $Z_5$, $Z_2$, $Z_2$, $Z_2$ and one of  the $Z_2$ chief factors is Frattini, it follows that $d(H)=2$.
\end{proof}

We can now suppose that $G={\Alt}_n$ or $G={\Sym}_n$ where $n\geq 5$.
We recall the O'Nan-Scott theorem which describes the maximal subgroups of $G$.  

\begin{thm}\label{t:onanscott} (\cite[Appendix]{AsSc}). 
Let $G={\Alt}_n$ or $G={\Sym}_n$ where $n\geq 5$. Let $H$ be a maximal subgroup of $G$.  One of the following assertions holds:
\begin{enumerate}[(i)]
\item $H$ is intransitive: $H=({\Sym}_k\times {\Sym}_{n-k})\cap G$ where $1\leq k < n/2$. 
\item $H$ is affine: $H={\AGL}_d(p)\cap G$ where $n=p^d$, $p$ is prime and $d\geq 1$.
\item $H$ is imprimitive or of wreath type: $H=({\Sym}_k \ {\wr}\  {\Sym}_t)\cap G$ where $n=kt$ or $n=k^t$ for some $t>1$.
\item $H$ is of diagonal type: $H=(T^k.({\Out}(T)\times {\Sym}_k))\cap G$ where $T$ is non-abelian simple and $n=|T|^{k-1}$ for some $k>1$.
\item $H$ is almost simple. 
\end{enumerate}
\end{thm}

We can now finish the proof of Proposition \ref{p:soclealt}.
\begin{proof}[{Proof of Proposition \ref{p:soclealt}}] By Proposition \ref{p:alsimple}, if $H$ is almost simple then $d(H)\leq 3$. We therefore assume in the remainder that $H$ is not almost simple. Suppose first that $G={\Sym}_n$. Suppose $H$ is intransitive. The chief factors of $H$ are then ${\Alt}_{k}$, ${\Alt}_{n-k}$, $Z_2$ and $Z_2$. In particular, if $A$ is a chief factor of $H$ then either $A$ is a nonabelian and $\delta_H(A)=1$, or $A\cong Z_2$ is a central chief factor of $H$ satisfying $\delta_H(A)\le 2$. It follows that $d(H)=2$.

Suppose now that $H$ is affine. Then $H=V:{\GL}_d(p)$ where $V=Z_p^d$ is an elementary abelian $p$-group and is the unique minimal normal subgroup of $H$.  Moreover $H=V.Z_a.{\Lf}_d(p).Z_b$. Since a chief factor of $H$ that is a section of $Z_a$ is Frattini, it follows from Lemma \ref{StApplication} that a non-Frattini chief factor $A$ of $H$ satisfies $\delta_H(A)=1$. In particular $d(H)=2$.

Suppose that $G$ is imprimitive or of wreath type. Then by Lemma \ref{StAnalysis} 
\begin{equation}\label{e:csimpwt}
H= \left\{    \begin{array}{ll}  {\Alt}_k^t.2.2^{t-2}.2.{\Alt}_t.2 & \textrm{if} \ t \ \textrm{is even} \\    {\Alt}_k^t.2.2^{t-1}.{\Alt}_t.2& \textrm{if} \ t \ \textrm{is odd}. \end{array}\right. 
\end{equation}
Moreover if $t$ is even then the first $Z_2$ chief factor of $H$ in (\ref{e:csimpwt}) is Frattini.

Suppose $t\neq 4$ and $k\neq 4$. By  Lemma \ref{StAnalysis}   (\ref{e:csimpwt}) is a chief series for $H$. 
 In particular, if $A$ is a chief factor of $H$ then  $\delta_H(A)\leq 2$, and if $\delta_H(A)=2$ then $A\cong Z_2$ is central. Therefore $d(H)=2$.
 
Suppose $t=4$ and $k\neq 4$.  By  Lemma \ref{StAnalysis}  the chief series for $H$ is 
\begin{equation*}
 {\Alt}_k^4.2.2^{2}.2.2^2.3.2.
\end{equation*}
Now $H$ has two chief factors  $A_1$ and $A_2$ with  $A_1\cong A_2\cong Z_2\times Z_2$. However $A_1$ and $A_2$ are not $H$-isomorphic and so they are  not $H$-equivalent. It follows that $d(H)=2$ also in this case.

Suppose $t\neq 4$ and $k=4$. By Lemma \ref{StAnalysis} the chief series of $H$ is 
\begin{equation*}
H= \left\{    \begin{array}{ll} 2^{2t}.3^t.2.2^{t-2}.2.{\Alt}_t.2 & \textrm{if} \ t \ \textrm{is even} \\     2^{2t}.3^t.2.2^{t-1}.{\Alt}_t.2& \textrm{if} \ t \ \textrm{is odd}. \end{array}\right. 
\end{equation*}
Again $d(H)=2$.\\
Suppose $t=k=4$. By  Lemma \ref{StAnalysis}  the chief series of $H$ is 
\begin{equation*}
H=  2^8.3^4.2.2^{2}.2.2^2.3.2 
\end{equation*}
Again, since the first $Z_2$ chief factor of $H$ is Frattini and the two $Z_2\times Z_2$ chief factors  of $H$ are not $H$-equivalent, it follows that $d(H)=2$.

Suppose next that $H$ is of diagonal type so that $H=T^k.({\Out}(T)\times {\Sym}_k)$ where $n=|T|^{k-1}$.  Now clearly $d(H)>1$. Since $T^k$ is a chief factor of $H$ and $\delta_H(T^k)=1$, we can restrict our analysis to the chief factors of $H$ which appear as sections of ${\Out}(T)\times {\Sym}_k$.  It follows from Proposition \ref{p:alsimple} that $d(H)\leq 4$. Moreover $d(H)=4$ if and only if $d({\Out}(T)\cap H/T^k)=3$. 

Suppose finally that $G={\Alt}_n$.  Note that  $H$ is a subgroup of index at most  2 of a  subgroup of ${\Sym}_n$ that is of intransitive, affine, imprimitive, wreath, diagonal,  or almost simple type as in  (i)-(v)  of Theorem  \ref{t:onanscott}. It now follows  from Lemma \ref{l:index2sub2gen} that if $H$ is not of diagonal type then $d(H)\leq 3$. Assume that $H$ is of diagonal type.  There are two cases to consider according respectively as  $H$ is a subgroup of $T^k.({\Out}(T)\times {\Sym}_k)$ of index 1 or 2. In the former case, as before, $d(H)\leq 4$ and $d(H)=4$ if and only if $d({\Out}(T)\cap H/T^k)=3$. 
We therefore assume that $H$ is a subgroup of $T^k.({\Out}(T)\times {\Sym}_k)$ of index $2$.  Note that $T^k$ is the unique minimal normal subgroup of $H$.
Let $D=\{(t_1,\dots,t_k):t_1\in T, t_1=\dots =t_k\}\cong T$ be the diagonal subgroup of $T^k$. Then ${\Sym}_n$ is the symmetric group on the set $\Omega$  of cosets of $D$ in $T^k$. Also $\sigma \in {\Sym_k}$ acts on $\Omega$ by sending a coset $D(t_1,\dots,t_k)$ to $D(t_{\sigma(1)},\dots, t_{\sigma(k)})$. 

Suppose that $k \geq 3$. A transposition $\tau$ of ${\Sym}_k$ then fixes $|T|^{k-2}$ points of $\Omega$. It follows that, seen as an element of ${\Sym}_n$, $\tau$ is a product of  $$\frac{n-|T|^{k-2}}2=\frac{|T|^{k-2}(|T|-1)}2$$ disjoint 2-cycles. In particular  a transposition of ${\Sym}_k$ is an even permutation of ${\Sym}_n$ and  ${\Sym_k}\leqs {\Alt}_n$. 
Therefore $H=T^k.(L\times {\Sym}_k)$ where $L$ is a subgroup of ${\Out}(T)$ of index 2. It follows from Proposition \ref{p:alsimple} that $d(H)\leq 4$. Moreover $d(H)=4$ if and only if $d({\Out}(T)\cap H/T^k)=3$. 

Suppose finally that $k=2$.  The transposition $\tau=(1,2)\in {\Sym}_k$ fixes a coset $D(t_1,t_2)$ if and only if $t_1=tt_2$ for some $t\in T$ of order dividing 2. In particular, $\tau$ fixes $i_2(T)+1$ points of $\Omega$ where $i_2(T)$ is the number of involutions of $T$. It follows that, seen as an element of ${\Sym}_n$, $\tau$ is a product of 
$$N=(|T|-i_2(T)-1)/2$$ 2-cycles.  
 
If $N$ is odd then $\tau \not \in H$, $H=T.{\Out}(T)$ and so, by Proposition \ref{p:alsimple},  $d(H)\leq 3$. If $N$ is even then $H=T.(L \times Z_2)$  where $L$ is a subgroup of ${\Out}(T)$ of index 2. It follows from Proposition \ref{p:alsimple} that $d(H)\leq 4$. Moreover $d(H)=4$ if and only if $d({\Out}(T)\cap H/T^k)=3$.\end{proof} 

\section{Almost simple groups with classical socle}\label{Classical}
In this section, we consider the classical case. To condense our arguments, we will first fix some notation. To keep consistent with \cite{KL}, which will be our main reference, we will write $\overline{\Omega}$ in place of $G_0$ throughout. We will also write $\overline{G}$ instead of $G$ (though we will seldomly refer to $\overline{G}$), and $V$ in place of the chief factor $A$, to avoid confusion with the automorphism group of $\overline{\Omega}$. Here, $\Omega$ is a certain normal subgroup of the group $I\le \GL_n(q)$ of isometries of a bilinear or quadratic form $\kappa$ on a vector space $\mathcal{V}$ of dimension $n$ over a field $\mathbb{F}$ of order $q^u$. We will write $q=p^f$, where $p$ is prime. Here, $u$ is defined to be $2$ if $\kappa$ is unitary (which will be referred to as case $\Uf$), and $u:=1$ if $\kappa$ is either identically zero (case $\Lf$), symplectic (case $\Sf$), or orthogonal (cases $\Of^{+}$, $\Of^{-}$ and $\Of^{0}$). The symbols $+$, $-$ and $0$ in the orthogonal cases refer to the Witt index of $\kappa$. For a more detailed discussion of these forms, and for a more precise definition of $\Omega$, see \cite[Chapter 2]{KL}.

With our reference still being \cite[Chapter 2]{KL}, we have a chain of groups
$$\Omega\le S\le I\le \Delta\le \Gamma\le A.$$
each being normalised by the last group $A$. Writing bars to denote reduction modulo the group $Z$ of scalars (which is consistent with our previous use of the bar notation), we get an $\overline{A}$-normal series 
$$\overline{\Omega}\le \overline{S}\le \overline{I}\le \overline{\Delta}\le \overline{\Gamma}\le \overline{A}.$$
For a subgroup $K$ of $A$, and a symbol $X\in\{\Omega,S,I,\Delta,\Gamma,A\}$, we will write $K_X$ [respectively $K_{\overline{X}}$], to denote the group $K\cap X$ [resp. $\overline{K\cap X}$]. With this in mind, we will write $G$ for the unique subgroup of $\Aut(A)$ containing $Z$, such that $\overline{G}=G/Z$. (We caution the reader again that this notation is unique to this section, and is not used elsewhere).

Here, $\overline{\Omega}$ is our \emph{simple classical group}. Furthermore, $\overline{A}=\Aut(\overline{\Omega})$, except that
\begin{rmk}\label{AutRemark} \begin{align*} |\Aut(\overline{\Omega}):\overline{A}| &=2\text{, in case $\Sf$, with $q$ even and $n=4$, or}\\
 |\Aut(\overline{\Omega}):\overline{A}|&=3\text{, in case $\Of^+$, with $n=8$ and $q$ odd.}\end{align*}\end{rmk}
Fix a maximal subgroup $H$ of $G$. We first deal with some cases of small dimension, which include the exceptional cases from Remark \ref{AutRemark}.
\begin{prop}\label{AutEx} Suppose that $n\le 12$. Then Theorem \ref{CrownLanguage} holds.\end{prop}
\begin{proof} Here, the maximal subgroups of the almost simple group $\overline{G}$ with socle $\overline{\Omega}$ are given in the tables in \cite[Chapter 8]{BHRD}. One can quickly read off from these tables that  $\delta_{\ol{H}}(V)\le 1$ [respectively $2$, $3$] if $V$ is non-abelian [resp. non-central, central] for any non-Frattini chief factor $V$ of $\overline{H}$. The result follows.\end{proof}  
By Remark \ref{AutRemark} and Proposition \ref{AutEx} we can and do assume, for the remainder of this section, that
\begin{align}\label{Assumption} \overline{A}=\Aut(\overline{\Omega}).\end{align}
Hence, $H=H\cap A$, and we have normal series'
\begin{align}\label{NormalSeries}H_{\Omega}\le H_{S}\le H_{I}\le H_{\Delta}\le H_{\Gamma}\le H_{A}=H\text{, and }
H_{\overline{\Omega}}\le H_{\overline{S}}\le H_{\overline{I}}\le H_{\overline{\Delta}}\le H_{\overline{\Gamma}}\le H_{\overline{A}}=\overline{H}.\end{align}
Fix a non-Frattini chief factor $V$ of $\overline{H}$. By Lemma \ref{ElementaryLemma} Part (i), we have 
\begin{align}\label{MainSeriesBound} \delta_{\overline{H}}(V)=\delta_{\overline{H},H_{\overline{\Omega}}}(V)+ \delta_{\overline{H},H_{\overline{S}}/H_{\overline{\Omega}}}(V)+ \delta_{\overline{H},H_{\overline{I}}/H_{\overline{S}}}(V) + \delta_{\overline{H},H_{\overline{\Delta}}/H_{\overline{I}}}(V)+\delta_{\overline{H},H_{\overline{\Gamma}}/H_{\overline{\Delta}}}(V)+\delta_{\overline{H},\overline{H}/H_{\overline{\Gamma}}}(V).\end{align}

We now proceed to the proof of Theorem \ref{CrownLanguage}. Our strategy will be to analyse the structure of $\overline{H}$ using Part (I) of the Main Theorem in \cite{KL}, and then apply the bound at (\ref{MainSeriesBound}). 

More precisely, the group $\overline{H}$ is either almost simple, or $H$ lies in one of eight natural classes of subgroups of $A$. This was proved, and the eight classes $\mathcal{C}_i$, for $1\le i\le 8$, were defined, in \cite{Asch}. In this paper, we use the definitions of $\mathcal{C}_i$ from \cite{KL}. If $\overline{H}$ is almost simple, then Theorem \ref{CrownLanguage} follows immediately from Proposition \ref{p:alsimple}, so we will assume that $H$ lies in one of the classes $\mathcal{C}_i$ for $1\le i\le 8$. We subdivide our proof accordingly, but first we require a preliminary lemma. Its proof follows easily from the analysis of the structure of outer automorphism groups of classical simple groups in \cite[Chapter 2]{KL}. 

\begin{lem}\label{KeyLemma} Let $V$ be a non-Frattini chief factor of $\overline{H}$. \begin{enumerate}[(i)]
\item If $V$ is non-abelian, or if $\dim_{\End_{\ol{H}}(V)}{V}>1$, then $\delta_{\overline{H}}(V)= \delta_{H_{\overline{\Omega}}}(V)$.
\item Assume that $\Omega$ is of type $\mathcal{T}$, where $\mathcal{T}\in\{\Lf,\Uf,\Sf,\Of^{\epsilon}\}$. If $V$ is abelian but non-central, then $\delta_{\overline{H}}(V)\le \delta_{H_{\overline{\Omega}}}(V)+f_{\mathcal{T}}(V)$, where $f_{\mathcal{T}}(V)=0$ if $\mathcal{T}\in \{\Sf,\Of^{\epsilon}\}$ and $f_{\mathcal{T}}(V)\le 1$ if $\mathcal{T}\in \{\Lf,\Uf\}$.
\item Assume that $\Omega$ is of type $\mathcal{T}$, where $\mathcal{T}\in\{\Lf,\Uf,\Sf,\Of^{\epsilon}\}$, and that $V$ is a central chief factor. Then $\delta_{\overline{H}}(V)\le \delta_{H_{\overline{\Omega}}}(V)+2$, unless $|V|=2$ and $H/H_{\overline{\Omega}}$ has an elementary abelian factor group of order $2^3$.
In these cases, $\delta_{\overline{H}}(V)= \delta_{H_{\overline{\Omega}}}(V)+3$.\end{enumerate}\end{lem}

We now proceed to proving Theorem \ref{CrownLanguage} in each of the Aschbacher classes.
\begin{prop}\label{C1} Suppose that $H$ lies in class $\mathcal{C}_1$, and that $H$ is non-parabolic. Then Theorem \ref{CrownLanguage} holds.\end{prop}
\begin{proof} In this case, $H$ stabilises a direct sum decomposition $\Vt=V_1\oplus V_2$, and $n_1:=\dim{V_1}\neq n_2:= \dim{V_2}$. Write $N_{X}(V_1,V_2)$ for the full stabiliser of $V_1\oplus V_2$ in $X$, as $X$ ranges over the symbols $\Omega$, $S$, $I$, $\Delta$, $\Gamma$ and $A$. Hence we have $N_{\Omega}(V_1,V_2)=H_\Omega\le H_I\le N_{I}(V_1,V_2)$. 
 
Let $I_i$ denote the group of isometries of $V_i$. Similarly define $\Omega_i$ and $\Delta_i$. By \cite[Lemma 4.1.1]{KL}, we have $N_{I}(V_1,V_2)={I_1}\times {I_2}\unlhd N_A(V_1,V_2)$, and $\Omega_1\times \Omega_2\le N_{\Omega}(V_1,V_2)=H_{\Omega}$. Then $L:=\Omega_1\times \Omega_2$ is characteristic in $H_{I}$, so $L$ is normal in $H$. Hence, we have a normal series for $\overline{H}$ of the form
\begin{align}\label{C1Series} 1<\overline{L}\le H_{\overline{\Omega}}\le \overline{H}.\end{align}
Thus, by Lemma \ref{ElementaryLemma} Part (i) we have
\begin{align}\label{C1SeriesBound} \delta_{\overline{H}}(V)=\delta_{\overline{H},\overline{L}}(V)+\delta_{\overline{H},H_{\overline{\Omega}}/\overline{L}}(V)+\delta_{\overline{H},\overline{H}/H_{\overline{\Omega}}}(V).\end{align}

We now proceed to bound each of the quantities on the right hand side of (\ref{C1SeriesBound}). We first consider the group $\overline{L}\cong \overline{\Omega_1}\circ \overline{\Omega_2}$. Suppose first that $\ol{\Omega_1}$ and $\ol{\Omega_2}$ are non-abelian simple. Then $\delta_{\overline{H},\overline{L}}(V)\le 1\text{ and }\delta_{\overline{H},\overline{L}}(V)=0\text{ if }V\text{ is abelian}$ by Lemma \ref{QuasisimpleLemma}. Next, assume that $\ol{\Omega_1}$ is non-simple. Then the possibilities for $\ol{\Omega_1}$ are listed in \cite[Proposition 2.9.2]{KL}. Since we are assuming that $n=n_1+n_2\ge 13$, we deduce that the group $\ol{\Omega_2}$ is simple. Moreover, $\delta_{\ol{\Omega_1}}(V)\le 1$ unless $\overline{\Omega_1}\in\{U_3(2),\Omega^+_4(3)\}$. Furthermore, in these cases we have $\delta_{\ol{\Omega_1}}(V)\le 2$, with equality if and only if $|V|=q$ and $V$ is central in $\ol{\Omega_1}$. Also, in this case any non-trivial diagonal automorphism of $\ol{\Omega_1}$ of even order permutes these two chief factors transitively. It follows that 
\begin{align}\label{L1}\delta_{\overline{H},\overline{L}}(V)\le 2\end{align}
with equality if and only if $\overline{\Omega_1}\in\{U_3(2),\Omega^+_4(3)\}$, $|V|=q$, and $\ol{H}/H_{\ol{\Omega}}$ is trivial.

Next, $H_{\overline{\Omega}}/\overline{L}$ is isomorphic to a subgroup of a direct product $I_1/\Omega_1\times I_2/\Omega_2\cong \mathbb{F}_q^{\times}\times \mathbb{F}_q^{\times}$ of cyclic groups. In fact, it is not difficult to see that 
$$H_{\overline{\Omega}}/\overline{L}\le \{(\alpha,\alpha^{-1})\text{ : }\alpha\in Q\},$$
where $Q:=2^2$ in case $\Of^{\pm}$ with $q$ odd, and $Q:=\mathbb{F}^{\times}$ otherwise. In particular, $H_{\overline{\Omega}}/\overline{L}$ is either isomorphic to $2^2$ (which can only happen if we are in case $\Of^{\pm}$ with $q$ odd), or $H_{\overline{\Omega}}/\overline{L}$ is cyclic. Hence, 
\begin{align}\label{L2}\delta_{\overline{H},H_{\overline{\Omega}}/\overline{L}}(V)\le 2\end{align}
with equality if and only if we are in case $\Of^{\pm}$ with $q$ odd, $H_{\overline{\Omega}}/\overline{L}\cong 2^2$, and $V$ is central in $\ol{L}$. Furthermore, if we are not in this case, then $\delta_{\overline{H},H_{\overline{\Omega}}/\overline{L}}(V)=0$ if $|V|$ is not prime.

Finally, Lemma \ref{KeyLemma} implies that $\delta_{\overline{H},\overline{H}/H_{\overline{\Omega}}}(V)\le 3$, with equality if and only if $|V|=2$ and $\overline{H}/H_{\overline{\Omega}}$ has an elementary abelian factor group of order $2^3$. Furthermore, $\delta_{\overline{H},\overline{H}/H_{\overline{\Omega}}}(V)=0$ if $|V|$ is not prime, by Lemma \ref{KeyLemma}.

Putting the information from the last three paragraphs together, and applying (\ref{C1SeriesBound}), we get $\delta_{\overline{H}}(V)\le 1$ [respectively $2$] if $V$ is non-abelian [resp. non-central]. Furthermore, if $V$ is central with $|V|>2$, then $\delta_{\overline{H}}(V)\le 3$.

So we may assume that $|V|=2$. Note that, by \cite[Proposition 2.9.2]{KL} and direct computation, any central chief factor in any non-simple $\overline{\Omega_i}$ is not centralised by any element of $\Aut(\ol{\Omega_i})\backslash \ol{\Omega_i}$. Hence, using the information from Paragraphs 3,4 and 5 above, and (\ref{C1SeriesBound}), we have $\delta_{\overline{H}}(V)\le 5$, with $\delta_{\overline{H}}(V)\geq 4$ if and only if one of the following holds:
\begin{enumerate}[(1)]
\item Case $\Lf$ with $q$ odd, $f$ even, and $\overline{H}/H_{\overline{\Omega}}$ elementary abelian of order $2^3$. Here, we have $\delta_{\overline{H}}(V)= 4$. This is because $H_{\overline{\Omega}}/L$ has a factor group of order $2$ in this case, which must therefore be centralised by $\ol{H}$.
\item Case $\Of^{\pm}$ with $q$ odd, $f$ even, $H_{\overline{\Omega}}/\overline{L}\cong 2$ or $2^2$, and $d(\overline{H}/H_{\overline{\Omega}})\ge 2$. This is because in this case the group $H_{\overline{\Omega}}/\overline{L}$ described in Paragraph 2, which is isomorphic to either $2$ or $2^2$, is centralised by any element of order $2$ in $\Out(\overline{\Omega})$. Hence we have $\ol{H}/\ol{L}\cong K_1\times K_2$, where $K_1$ is a $2$-group with $d(K_1)=3$ or $4$, and $K_2$ cyclic of order $f$. Furthermore, the only field automorphism centralising the group $H_{\overline{\Omega}}/\overline{L}$ is the involution. Hence, we must have $K_2=1$ or $K_2=2$.\end{enumerate}
Thus we have $\delta_{\ol{H}}(V)\le 3$ except when $|V|=2$ and cases (1) or (2) above hold. This completes the proof.\end{proof}
  
\begin{prop}\label{C1Parabolic} Suppose that $H$ lies in class $\mathcal{C}_1$, and that $H$ is parabolic. Then Theorem \ref{CrownLanguage} holds.\end{prop}
\begin{proof} In this case, $H$ stabilises a totally singular subspace $W\subset \Vt$ of dimension $m<\frac{n}{2}$ (or $m\le\frac{n}{2}$ in case $\Lf$). Write $N_{X}(W)$ for the full stabiliser of $W$ in $X$, as $X$ ranges over the symbols $\Omega$, $S$, $I$, $\Delta$, $\Gamma$ and $A$. Similarly define the centralisers $C_{X}(W)\le N_X(W)$. Hence we have $N_{\Omega}(W)=H_\Omega\le H_I\le N_{I}(W)$. 
 
By \cite[Lemmas 4.1.12 and 4.1.13]{KL}, there exists subspaces $U$ and $Y$ of $\Vt$, with $Y$ totally singular, such that $\Vt=(W\oplus Y)\perp U$ ($U=0$ in case $L$), and $N_I(W)=C\rtimes L$, where 
$$C:=C_{I}(W)\cap C_{I}(W^{\perp}/W)\cap C_{I}(\Vt/W^{\perp})\text{ and }L:=N_{I}(W)\cap N_{I}(Y)\cap N_I(U).$$ Moreover, $\dim{Y}=m$ unless we are in case $\Lf$, in which case $\dim{Y}=n-m$. It is an easy exercise to show that $C$ is a nilpotent characteristic subgroup of $H_{\Omega}$, that $C/C'$ is an elementary abelian $p$-group, and that $C$ is contained in $N_\Omega(W)$. In particular, $C'\le \Frat(C)$. Furthermore, by \cite[Theorem 2]{ABS} $C/C'$ is either irreducible or has composition length $2$, with non-isomorphic composition factors. It follows from Clifford's Theorem that $C/C'$ is an $\mathbb{F}_p[\overline{H}]$-module, with composition length at most $2$, and if it has composition length $2$, then the factors are non-isomorphic.

We now consider the structure of $L$: let $I(U)$ denote the group of $\kappa$-isometries of the $(n-2m)$-dimensional vector space $U$. Similarly define $\Omega(U)$ In cases $\Uf$, $\Sf$ and $\Of$ we have, by \cite[Lemma 4.1.12]{KL}, that $L\cong \GL_m(q^u)\times I(U)$. In case $\Lf$ we have, by \cite[Lemma 4.1.1]{KL}, that $L\cong \GL_{m}(q)\times \GL_{n-m}(q)$. Let $M:=\SL_{m}(q)\times \SL_{n-m}(q)\unlhd L$ in case $\Lf$, and let $M$ be the normal subgroup $\SL_m(q^u)\times \Omega(U)$ of $L$ otherwise. Then $M\le N_{\Omega}(W)\le L_{\Omega}\le H_{\Omega}$. Furthermore, since $M$ is characteristic in $L\unlhd N_A(W)$, it is normal in $N_A(W)$. Hence, $M$ is normal in $H$. Thus, we have a normal series
\begin{align}\label{C1PSeries} 1< C \le C\rtimes \overline{M}\le C\rtimes L_{\overline{\Omega}}= H_{\overline{\Omega}}\le\overline{H}\end{align}
for $\overline{H}$. Thus, by Lemma \ref{ElementaryLemma} we have
\begin{align}\label{C1PSeriesBound} \delta_{\overline{H}}(V)=\delta_{\overline{H},\overline{C}}(V)+\delta_{\overline{H},\overline{M}}(V)+\delta_{\overline{H},L_{\overline{\Omega}}/\overline{M}}(V)+\delta_{\overline{H},\overline{H}/H_{\overline{\Omega}}}(V).\end{align} 
As in the proof of Proposition \ref{C1}, we now proceed to bound each of the quantities on the right hand side of (\ref{C1PSeriesBound}).

First, the dimensions $a_i$ (for $i\le 2$) over $\mathbb{F}_p$ of the irreducible $\mathbb{F}_p[\overline{H}]$-components of the module $C/C'$ are calculated in \cite[Propositions 4.1.17, 4.1.18, 4.1.19 and 4.1.22]{KL}: $a_i$ is either $m(n-m)$ or $2mn-3m^2$ in case $\Lf$; $a_i=m(2n-3m)$ for each $i$ in case $\Uf$; $a_i=\frac{m}{2}-\frac{3m^2}{2}+mn$ for each $i$ in case $\Sf$; and $a_i=mn-\frac{m}{2}(3m+1)$ for each $i$ in case $\Of$. In particular, $\delta_{\ol{H},C}(V)\le 1$, and $\delta_{\ol{H},C}(V)=0$ if either $V$ is not a $p$-group, or $V$ is a $p$-group with $\dim_{\mathbb{F}_p}(V)$ not equal to any the dimensions $a_i$ above. 

Now, from (\ref{C1PSeriesBound}), $\ol{H}/C$ has a normal series identical to the normal series in (\ref{C1Series}) in the proof of Proposition \ref{C1}. From this proof, one quickly sees that no chief factor of $\ol{H}/C$ can have dimension greater than $2$. Thus, by the paragraph above we have $\delta_{\ol{H}}(C)=1$, and for any non-Frattini chief factor $V\neq C$ of $\ol{H}$, we have $\delta_H(V)\le 1$ [respectively $2$] if $V$ is non-abelian [resp. non-central]. If $V\neq C$ is central, then $\delta_H(V)\le 5$, with $\delta_H(V)\geq 4$ if and only if one of the cases (1) or (2) occurs as in the proof of Proposition \ref{C1}. This completes the proof.\end{proof}  

\begin{prop}\label{C2} Suppose that $H$ lies in class $\mathcal{C}_2$. Then Theorem \ref{CrownLanguage} holds.\end{prop}
\begin{proof} In this case, $H$ is the stabiliser in $G$ of a subspace decomposition
$$\Vt=V_1\oplus V_2\oplus\hdots\oplus V_t,$$
where $V_i$ is either a non-degenerate or totally singular $m$-space for all $i$ (the dimension $m$ is fixed), and $V_i$ is orthogonal to $V_j$, for each $i\neq j$. Write $\mathcal{D}:=\{V_1,\hdots,V_t\}$, and let $I_i$ denote the group of $\kappa$-isometries of $V_i$. Similarly define $\Omega_i$ and $\Delta_i$. Also, write $X_{\mathcal{D}}$ for the full stabiliser of $\mathcal{D}$ in $X$, as $X$ ranges over the symbols $\Omega$, $S$, $I$, $\Delta$, $\Gamma$ and $A$. Denote the kernel of the action of $X_{\mathcal{D}}$ on $\mathcal{D}$ by $X_{(\mathcal{D})}\unlhd X_\mathcal{D}$, and the induced action of $X_{\mathcal{D}}$ on $\mathcal{D}$ by $X^{\mathcal{D}}$. In particular, $X^{\mathcal{D}}\cong X_\mathcal{D}/X_{(\mathcal{D})}\le \Sym_t$. 

Since $\Omega\le G$, we have $\Omega_{\mathcal{D}}\le G_{\mathcal{D}}=H$. Hence, $\Omega_{\mathcal{D}}=H_\Omega\le H_\Delta\le \Delta_\mathcal{D}$, and we have a normal series
\begin{align}\label{C2Series} 1<H_{\ol{\Omega_{(\mathcal{D})}}}\le H_{\ol{\Omega}}\le \ol{H}.\end{align}
Thus by Lemma \ref{ElementaryLemma} we have
\begin{align}\label{C2SeriesBound} \delta_{\ol{H}}(V)= \delta_{\ol{H},H_{\ol{\Omega_{(\mathcal{D})}}}}(V)+\delta_{\ol{H},H_{\ol{\Omega^{\mathcal{D}}}}}(V)+\delta_{\ol{H},\ol{H}/H_{\ol{\Omega}}}(V).\end{align}  

Now, the members of $\mathcal{C}_2$ are distinguished into types in \cite[Table 4.2.A]{KL}, and we divide our proof accordingly.

Suppose first that $H$ is of type $\GL_m(q)\wr\Sym_t$, $\GU_m(q)\wr\Sym_t$, $\Sp_m(q)\wr\Sym_t$ or $O^{\epsilon}_m(q)\wr\Sym_t$. Then the spaces $V_i$ are mutually isometric, and $t>2$. Moreover, \cite[Lemma 4.2.8 Part (iii)]{KL} implies that
$$H_{\ol{\Omega}}=\ol{\Omega_\mathcal{D}}=\ol{\Omega_{(\mathcal{D})}}\ol{J}\le \PDelta_1\wr {\Sym}_t,$$
where $J_{\overline{\Omega}}=\Sym_t$, apart from in case $\Of$ with $m=1$ and $q= \pm 3$ (mod $8$). In this case, $J_{\overline{\Omega}}=\Alt_t$, and \begin{enumerate}[(a)]
\item $\POmega/\POmega$ has abelian Frattini quotient;
\item $\ol{\Omega_{(\mathcal{D})}}\geq \POmega_1^t$; and
\item $H_{\ol{\Omega_{(\mathcal{D})}}}/\POmega_1^t$ is the fully deleted subgroup of $(\PDelta_1/\POmega_1)^t$.\end{enumerate}
Finally, since $N_H(\PDelta_1)$ is precisely the stabiliser in $H$ of $V_1$, \cite[Lemma 4.2.1]{KL} implies that $\ol{H}=N_{\overline{H}}(\PDelta_1)H_{\overline{\Omega}}$. 

Thus, we can use Lemma \ref{StAnalysis}, together with Lemma \ref{StApplication}, to find the chief factors of $\ol{H}$ contained in $H_{\ol{\Omega_{(\mathcal{D})}}}$. Indeed, note that the outer automorphism group $T$ induced by the action of $N_{\ol{H}}(\POmega_1)$ on $\POmega_1$ is isomorphic to the outer automorphism group induced by $\ol{H}$ on $\ol{\Omega}$.
Hence, by Lemma \ref{StAnalysis}, apart from case $\Of$ with $q$ odd, there is, for each prime $r$ a unique non-Frattini chief factor of $\ol{H}$ contained in $H_{\ol{\Omega_{(\mathcal{D})}}}/\POmega_1^t$: In the case $r\mid t$, this has order $r^{t-2}$ (recall that $t>2$). Otherwise, it has order $r^{t-1}$. 

Next, assume that case $\Of$ holds, with $q$ odd. If $H_{\ol{\Delta}}/H_{\ol{\Omega}}\cong D_8$, then $T$ permutes the two chief factors of order $2$ of $H_{\ol{I}}/H_{\ol{\Omega}}$. Otherwise, $T$ fixes them. Hence, if $H_{\ol{\Delta}}/H_{\ol{\Omega}}< D_8$ then there are two non-Frattini chief factors of $\ol{H}$ contained in $H_{\ol{\Omega_{(\mathcal{D})}}}/\POmega_1^t$: in the case $r\mid t$, both have order $2^{t-2}$. Otherwise, both have order $2^{t-1}$. If $H_{\ol{\Delta}}/H_{\ol{\Omega}}\cong D_8$ then there is one non-Frattini chief factor of $\ol{H}$ contained in $H_{\ol{\Omega_{(\mathcal{D})}}}/\POmega_1^t$: in the case $r\mid t$,it has order $2^{2(t-2)}$. Otherwise, both have order $2^{2(t-1)}$. In any case, $H_{\ol{\Omega_{(\mathcal{D})}}}/\POmega_1^t$ comprises either one or two non-central non-Frattini chief factors of $\ol{H}$.

Next, we consider the non-Frattini chief factors of $\ol{H}$ contained in $\POmega_1^t$. If $\POmega_1$ is simple, then $\POmega_1^t$ is a non-abelian chief factor of $\ol{H}$, by Lemma \ref{StAnalysis}. If $\POmega_1=\Sp_4(2)$, then $\POmega_1^t$ comprises three non-Frattini chief factors of $\ol{H}$: the non-abelian $\Alt_6^t$; a non-central abelian of order $2^{t-(2,t)}$; and a central, of order $2$, again by Lemma \ref{StAnalysis}. Finally, if $\POmega_1$ non-simple, then $\POmega_1$ is listed in \cite[Proposition 2.9.2]{KL}. In particular, by Lemma \ref{StApplication} we have $\delta_{\PDelta_1,\POmega_1}(W)\le 1$ for any non-Frattini chief factor $W$ of $\PDelta_1$, with equality if and only if $W$ is non-central and contained in $\PDelta_1$. It follows, again using Lemma \ref{StAnalysis}, that for each non-Frattini chief factor $W$ of $\PDelta_1$ contained in $\POmega_1$, we get a unique non-Frattini chief factor, $\ol{H}$-equivalent to $W^t$, contained in $\POmega_1^t$. This is non-central, again by Lemma \ref{StAnalysis}. 

In summary, we have $\delta_{\ol{H},H_{\ol{\Omega_{(\mathcal{D})}}}}(V)\le 1$, with equality possible only if either $V$ is non-central, or $|V|=2$ and $\POmega_1=\Sp_4(2)$.  

Thus, we have determined the quantity $\delta_{\ol{H},H_{\ol{\Omega_{(\mathcal{D})}}}}(V)$; we now determine the other quantities in the bound (\ref{C2SeriesBound}). Note that $\ol{H}/H_{\ol{\Omega^{\mathcal{D}}}}\cong J.(\ol{H}/H_{\ol{\Omega}})$, and $J=\Sym_t$, unless we are in case $\Of$ with $m=1$ and $q\sim \pm 3$ (mod $8$), in which case $J=\Alt_t$. Suppose that we are not in this latter case. Then the Frattini subgroup of $\ol{H}/H_{\ol{\Omega^{\mathcal{D}}}}$ is contained in the centraliser of the $\Alt_t$ normal subgroup, since $\Aut(\Alt_t)$ is elementary abelian. In particular, the $\ol{H}$-chief factor $Z_2$ in $J=\Sym_t$ is non-Frattini in $\ol{H}$. Thus, we conclude that $\delta_{\ol{H},H_{\ol{\Omega^{\mathcal{D}}}}}(V)\le 1$, with equality if and only if either ($|V|=2$ [in particular, $V$ is central] and we are not in the case $\Of$ with $m=1$ and $q\sim \pm 3$ (mod $8$)); or $V\cong \Alt_t$. Finally, $\delta_{\ol{H},\ol{H}/H_{\ol{\Omega}}}(V)\le 3$, with equality if and only if $|V|=2$ and $\overline{H}/H_{\overline{\Omega}}$ has an elementary abelian factor group of order $2^3$. Here, we used Lemma \ref{KeyLemma}.

Next, assume that $m=\frac{n}{2}$ and that $H$ is of type $\GL_{\frac{n}{2}}(q^u).2$. Then $H_{I}$ has shape $\GL_{\frac{n}{2}}(q^u).2$ by \cite[Corollary 4.2.2 Part (ii) and Lemma 4.2.3]{KL}. Here, $H_{I}/\GL_{\frac{n}{2}}(q^u)=H_{I}^{\mathcal{D}}\cong \Sym_2$ is the induced action of $H_{I}$ on $\mathcal{D}$. If $n=4$ and $q^u\le 3$, then $(n,q)\in\{(4,2),(4,3)\}$ and the result follows from Proposition \ref{AutEx}. So assume that $\SL_{\frac{n}{2}}(q^u)$ is perfect. Then $Z\le \Frat(\SL_{\frac{n}{2}}(q^u))$. Hence, since $\GL_{\frac{n}{2}}(q^u)/\SL_{\frac{n}{2}}(q^u)$ is cyclic, we have $\delta_{H_{\overline{I}}}(V)\le 2$, with equality possible only if $|V|=2$ and $f$ is even. Since $\overline{H}/H_{\overline{I}}$ is metacyclic, it follows that $\delta_{\overline{H}}(V)\le 3$, except possibly when $|V|=2$ and $f$ is even. However, by \cite[Lemma 4.2.3]{KL}, in this case a field automorphism of order $2$ in $\Aut(\overline{\Omega})$ acts by interchanging $V_1$ and $V_2$. Thus, $\overline{H}/H_{\overline{I}}$ has at most one non-Frattini chief factor of order $2$. 

Finally, assume that $m=\frac{n}{2}$ and that $H$ is of type $O_{\frac{n}{2}}(q)^2$, so that $X\in \{O^{+},O^{-}\}$ and $mq$ is odd. By \cite[Proof of Proposition 4.2.16]{KL} we have $H_{\overline{I}}\cong O_m(q)\times O_m(q)=\Omega_m(q).2^2\times \Omega_m(q).2^2$, and $\soc{(H_{\overline{I}})}\cong \Omega_m(q)\times \Omega_m(q)$. Furthermore, By \cite[Lemma 4.2.2]{KL}, there exists an element $\sigma$ of $H_{\Delta}$ such that $V_1^\sigma=V_2$. Hence, $\sigma$ interchanges the two copies of $O_m(q)$ in $H_{\overline{I}}$, from which it follows that $H_{\overline{I}}$ contains precisely three $\overline{H}$-chief factors: $\soc{(H_{\overline{I}})}$, and two copies of $2^2$. Since all non-Frattini $\overline{H}$-chief factors contained in $\overline{H}/H_{\overline{I}}$ are central, we have $\delta_{\overline{H}}(V)= 1$ if $V$ is non-abelian; $\delta_{\overline{H}}(V)\le 3$ if $V$ is central; and $\delta_{\overline{H}}(V)\le 2$ if $V$ is abelian but non-central. 

It now follows that $\delta_H(V)\le 1$ [respectively $2$] if $V$ is non-abelian [resp. non-central]. If $V$ is central, then $\delta_H(V)\le 4$, with equality if and only if $|V|=2$, $t>2$, and either case $\Lf$ holds and $\ol{H}/H_{\ol{\Omega}}\cong 2^3$; or case $\Of^{\pm}$ holds and (either $m\neq 1$ or $m\sim\pm 1$ (mod $8$)) and $\ol{H}/H_{\ol{\Omega}}\cong 2^3$ or $D_8$. This completes the proof.\end{proof}

\begin{prop}\label{C3} Suppose that $H$ lies in $\mathcal{C}_3$. Then Theorem \ref{CrownLanguage} holds.\end{prop}
\begin{proof} This case is easier to handle than the previous ones. Indeed by \cite[Section 4.3]{KL}, $H_{\overline{\Omega}}$ has a characteristic quasisimple subgroup $L$, with $C_{\ol{H}}(L)=Z(L)$. In particular, since it follows that $L\unlhd \overline{H}$, we have, by Lemma \ref{QuasisimpleLemma}, that $Z(L)\le \Frat(L)\le \Frat(\overline{H})$. Thus, $\overline{H}/Z(L)$ is isomorphic to a subgroup of $\Aut(L/Z(L))$ containing $L/Z(L)$. In particular, $\delta_{\ol{H}}(V)\le 1$ [respectively $2$, $3$] if $V$ is non-abelian [resp. non-central, central] for any non-Frattini chief factor of $\overline{H}$, by Proposition \ref{p:alsimple}. This completes the proof.\end{proof}

\begin{prop}\label{C4} Suppose that $H$ lies in class $\mathcal{C}_4$. Then Theorem \ref{CrownLanguage} holds.\end{prop}
\begin{proof} Here, $H$ stabilises a tensor product decomposition $\Vt=V_1\otimes V_2$, and $\kappa=\kappa_1\otimes\kappa_2$, where $\kappa_i$ is a bilinear form on $V_i$. For more information on the possibilities for the $\kappa_i$, see \cite[Table 4.4.A]{KL}. Also, $(V_1,f_1)$ is not similar to $(V_2,f_2)$. Let $I_i$ denote the group of isometries of $(V_i,\kappa_i)$, and similarly define $\Omega_i$, $S_i$ and $\Delta_i$, and the projective equivalents $\PI_i$, $\POmega_i$, $\PS_i$ and $\PDelta_i$.  Also, write $n_i:=\dim_\mathbb{F}V_i$, so that $n=n_1n_2$. We have $H_{\ol{I}}\cong (\PI_1\times \PI_2)\langle z\rangle$, where $z\in H_{\ol{S}}$, and $z^2\in \PI_1\times \PI_2$ (see \cite[Lemma 4.4.5]{KL}).

Assume first that $H_{\overline{\Omega}}$ is non-local. Then $\soc{(H_{\overline{\Omega}})}=\POmega_1'\times \POmega_2'$, by \cite[Lemma 4.4.9]{KL}. Hence, $\ol{H}$ has a normal series
$$1<L\le H_{\overline{\Omega}}\le\ol{H}$$
where $L:=\soc({H_{\overline{\Omega}}})$ is a direct product of non-abelian simple groups, and $H_{\overline{\Omega}}/L$ is isomorphic to a subsection of
$$\{(\alpha,\alpha^{-1})\text{ : }\alpha\in Q\}.$$
where $Q:=2^2$ in case $\Of$; $Q:=2$ in case $I_i=\Sp_4(2)$; and $Q:=\mathbb{F}^{\times}$ otherwise. Furthermore, $H_{\overline{\Omega}}/L$ has even order if and only if $q$ is odd. It follows that $\delta_{\ol{H},H_{\overline{\Omega}}}(V)\le 1$ for any chief factor $V$ of $\overline{H}$, unless case $\Of^{\pm}$ holds with $q$ odd, $|V|=2$ and $H_{\ol{\Delta}}/H_{\overline{\Omega}}<D_8$. In this case, $\delta_{\ol{H},H_{\overline{\Omega}}}(V)=2$. Thus, $\delta_{\ol{H}}(V)\le 1$ [respectively $2$, $5$] if $V$ is non-abelian [resp. non-central, central] for any non-Frattini chief factor of $\overline{H}$. Furthermore, arguing as in the final paragraph of the proof of Proposition \ref{C1} we see that we get $\delta_{\ol{H}}(V)=4$ if and only if $|V|=2$, $q$ is odd, and either case $\Lf$ holds and $\ol{H}/H_{\ol{\Omega}}\cong 2^3$; or case $\Of^{\pm}$ holds and $\ol{H}/H_{\ol{\Omega}}\cong 2^2$ or $C_4\times C_2$. Also, we have $\delta_{\ol{H}}(V)=5$ if and only if $|V|=2$, $q$ is odd, and case $\Of^{\pm}$ holds with $\ol{H}/H_{\ol{\Omega}}\cong 2^3$ or $D_8\times C_2$. This gives us what we need. 

Hence, we may assume that $H_{\overline{\Omega}}$ is local. Then $\POmega_i$ is local for some $i$, by \cite[Proposition 4.4.9]{KL}. Hence, by [Table 4.4.A and Proposition 2.9.2]{KL}, we have $$I_i\in \{\GL_{2}(2),\GL_2(3),\GU_2(2),\GU_2(3),\GU_3(2),\Sp_2(3),O_3(3),O^{+}_4(3)\}.$$ In particular, $\POmega_i$ is soluble. If $n=n_1n_2\le 12$ then the result can be quickly checked using \cite{BHRD}. Thus we may assume, without loss of generality, that $\POmega_1'$ is a non-abelian simple group. Hence, since $H_I\le (I_1\circ I_2)\langle z\rangle$, we have 
\begin{align}\label{C4Structure}  H_{\Omega}=H_{I}\cap\Omega=(\Omega_1\circ\Omega_2).Z_a, \end{align}
 where $Z_a$ is isomorphic to a subsection of
$$\{(\alpha,\alpha^{-1})\text{ : }\alpha\in Q\}.$$
where $Q:=2^2$ in case $\Of^{\pm}$; $Q:=2$ in case $I_1=\Sp_4(2)$; and $Q:=\mathbb{F}^{\times}$ otherwise. Now, since $\POmega_2$ and the outer automorphism group of any simple group are soluble, it follows that $\delta_{\overline{H}}(V)\le 1$ if $V$ is non-abelian. 

Next, note that since $\POmega_2\ \charac \ H_{\ol{\Omega}}$, $C_{H_{\overline{\Omega}}}(\POmega_2)$ is normal in $\overline{H}$. Hence, $\overline{H}$ has a normal series
$$1\le C_{H_{\overline{\Omega}}}({\POmega_2})\le \overline{H}.$$
Furthermore, by (\ref{C4Structure}), we have $C_{H_{\overline{\Omega}}}({\POmega_2})\cong \POmega_1$. By noting that $C_{H_{\overline{\Omega}}}({\POmega_2})$ is either simple or isomorphic to $\Sym_6$, we have $\delta_{C_{H_{\overline{\Omega}}({\POmega_2})}}(V)\le 1$. The group $\overline{H}/C_{H_{\overline{\Omega}}}({\POmega_2})$ is a subgroup of $\Aut(\POmega_2)$, and so its chief factors are easily found by direct computation. It quickly follows that $\delta_{\ol{H}}(V)\le 1$ [respectively $2$, $3$] if $V$ is non-abelian [resp. non-central, central] for any non-Frattini chief factor of $\overline{H}$.\end{proof}

\begin{prop}\label{C5} Suppose that $H$ lies in class $\mathcal{C}_5$. Then Theorem \ref{CrownLanguage} holds.\end{prop}
\begin{proof} In this case, $H$ stabilises a vector space $V_{1}$ of dimension $n$ over a subfield $\mathbb{F}_1$ of $\mathbb{F}$ of prime degree $r$. Let $\kappa_{1}$ be the associated bilinear form over $\mathbb{F}_1$, and define $\Omega_{1}$, $\Delta_1$, $\POmega_1$ and $\PDelta_1$ in the usual way. Then $H_{\ol{\Gamma}}\le N_{\ol{\Gamma}}(V_{1})\cong \PDelta_1\rtimes \langle \phi_r\rangle$, where $\phi_r$ is a generator for $\Gal(\mathbb{F}:\mathbb{F}_1)$. It follows that the centraliser in $\ol{H}$ of $\ol{\Omega}_1$ is contained in $\langle \phi_r\rangle$, and hence that $H/H_{\overline{\Omega}}$ is isomorphic to a subgroup of $\Out({\POmega}_1).r$, where the extension is a cyclic extension of the associated group of field automorphisms of $\Omega_1$. Thus, $H/H_{\overline{\Omega}}$ is an extension of at most three cyclic groups, and has at most one non-central chief factor. The bound $\delta_{\ol{H}}(V)\le 1$ [respectively $2$, $3$] now follows immediately if $\POmega_{1}'$ is a non-abelian simple, and $V$ is non-abelian [resp. non-central, central]. Otherwise, the list of possibilities for $\POmega_1'$ is in \cite[Proposition 2.9.2]{KL}, and $\delta_{\ol{H}}(V)\le 1$ [resp. $2$, $3$]  follows quickly by direct computation.\end{proof}

\begin{prop}\label{C6} Suppose that $H$ lies in class $\mathcal{C}_6$. Then Theorem \ref{CrownLanguage} holds.\end{prop}
\begin{proof} Here, the group $H_{\overline{\Delta}}$ is a subgroup of $N_{\overline{\Delta}}(R)$, for an extra special $r$-group $R$, with $r$ prime (see \cite[Definition, Page 150]{KL}). By \cite[(4.6.1) and Table 4.6.A]{KL}, the group $N_{\overline{\Delta}}(R)$ has shape $R.C$, where $C\cong \Sp_{2m}(r)$, or $r=2$ and $C\cong O^{\pm}_{2m}(2)$. Also, $C$ acts naturally on $R$. Furthermore, $R$ is contained in $H_{\overline{\Omega}}$, and $H_{\overline{\Omega}}/R$ acts irreducibly on $R$. Since $H_{\overline{\Omega}}$ is normal in $R.C$, it follows that $H_{\overline{\Delta}}$ has shape $R.C_1$, where $C_1$ is a subgroup of $C$ containing a (non-trivial) normal irreducible subgroup of $C$. We also have $n=r^m$.

Suppose first that $n\geq 4$. Then $C/Z(C)$ is almost simple, and $C'$ is quasisimple. Furthermore, $C/C'$ is cyclic. It follows that $\delta_{\ol{H},H_{\overline{\Delta}}}(V)\le 1$ (recall that the centre of a quasisimple group is Frattini - see Lemma \ref{QuasisimpleLemma}). Since $\overline{H}/H_{\overline{\Delta}}$ is abelian of rank at most $2$, it follows that $\delta_{\ol{H}}(V)\le 1$ [respectively $2$, $3$] if $V$ is non-abelian [resp. non-central, central] for any non-Frattini chief factor of $\overline{H}$, which gives us what we need.  

Next, assume that $n=3$, so that $(r,m)=(3,1)$. Then $H_{\overline{\Delta}}$ has shape $3^2.C_1$, where $C_1\in\{Q_8,\Sp_2(3)\}$ (see \cite[Proof of Proposition 4.6.4]{KL}). It follows that $\delta_{\ol{H},H_{\overline{\Delta}}}(V)\le 1$, and the required bounds follow as in the first paragraph above. The case $m=2$ is similar.\end{proof}

\begin{prop}\label{C7} Suppose that $H$ lies in class $\mathcal{C}_7$. Then Theorem \ref{CrownLanguage} holds.\end{prop}
\begin{proof} Here, $H_\Gamma$ is a subgroup of the stabiliser in $\Gamma$ of a tensor decomposition
$$\Vt=V_1\otimes V_2\otimes\hdots\otimes V_t,$$
where $\kappa=\kappa_1\otimes\hdots\otimes \kappa_t$, for bilinear forms $\kappa_i$ on $V_i$. Furthermore, the spaces $(V_i,\kappa_i)$ are isometric, so they all have a common dimension $m$. In particular, $n=m^t$. For a more detailed description of the possible forms $\kappa_i$, see \cite[Section 4.7]{KL}. Write $I_i$ for the group of isometries of $(V_i,\kappa_i)$. Similarly define $\Omega_i$ and $\Delta_i$ and the projective equivalents $\POmega_i$, $\PDelta_i$ and $\PI_i$. 

Let $\mathcal{D}=\{V_1,\hdots,V_t\}$, and write ${\Gamma}_\mathcal{D}$ for the full (set-wise) stabiliser in $\Gamma$ of $\mathcal{D}$. Also, for any subgroup $K$ of ${\Gamma}_\mathcal{D}$, write $K_{(\mathcal{D})}$ for the kernel of the action of $K$ on $\mathcal{D}$. Also, write $K^\mathcal{D}=K/K_{(\Delta)}\le \Sym_t$ for the induced action of $K$ on $\mathcal{D}$. We then have a normal series
$$1\le H_{\overline{\Delta(\mathcal{D})}}\le H_{\overline{\Delta}}\le \overline{H}.$$
Thus, we have 
\begin{align}\label{7Bound} \delta_{\overline{H}}(V)= \delta_{ H_{\overline{\Omega(\mathcal{D})}}}(V)+ \delta_{H_{\overline{\Omega}/H_{\overline{\Omega(\mathcal{D})}}}}(V)+\delta_{\overline{H}/H_{\overline{\Omega}}}(V).\end{align}
Now, since $G$ contains $\Omega$, $H_{\Delta}$ contains the full set-wise stabiliser $\Omega_\mathcal{D}$ of $\mathcal{D}$ in $\Omega$. By \cite[(4.7.1) and (4.7.2)]{KL}, we have 
\begin{align}\label{Wreath} \overline{{\Delta}_{\mathcal{D}}}=\overline{\Delta_{(\mathcal{D})}}\overline{J}\cong \PDelta_1\wr{\Sym}_t \end{align}
(with the product action), where $\overline{\Delta_{(\mathcal{D})}}\cong (\PDelta_1)^t$, and $\overline{J}\cong \Sym_t$. Furthermore by \cite[(4.7.8)]{KL}, the intersection $J_{\overline{\Omega}}$ is $\Alt_t$ if either $t=2$ and $m=2$ (mod $4$), or $m=3$ (mod $4$). Otherwise, $J_{\overline{\Omega}}$ is $\Sym_t$. 

Next, $H_{\overline{\Omega_{(\mathcal{D})}}}$ is contained in $\PDelta_1^t$, and contains the subgroup $\POmega_1^t$. Moreover, $H_{\overline{\Omega_{(\mathcal{D})}}}/\POmega_1^t$, Also, $N_{\ol{H}}(\PDelta_1)=\Stab_{\ol{H}}(V_1)$. Hence, since $\ol{H}^{\mathcal{D}}$ is transitive, it follows that $\ol{H}=N_{\ol{H}}(\PDelta_1)H_{\ol{\Delta}}$. The other conditions required in Lemma \ref{StAnalysis} follow from Lemma \ref{StApplication}.

Thus, we may apply Lemma \ref{StAnalysis} to find the chief factors of $\ol{H}$ contained in $H_{\overline{\Omega_{(\mathcal{D})}}}$. In fact, if $t\neq 2$, the argument is exactly the same as in the first part of the proof of Proposition \ref{C2}, except that here the case $\POmega_1'$ local does not occur, by \cite[Lemma 4.7.1]{KL}. Thus, if $t\neq 2$ then we have $\delta_{\ol{H},H_{\ol{\Omega_{(\mathcal{D})}}}}(V)\le 1$, with equality possible only if either $V$ is non-central, or $|V|=2$ and $\POmega_1=\Sp_4(2)$.

If $t=2$, then we argue in the same way, but in this case we get $\delta_{\ol{H},H_{\ol{\Omega_{(\mathcal{D})}}}}(V)\le 1$ if $V$ is non-abelian or non-central; $\delta_{\ol{H},H_{\ol{\Omega_{(\mathcal{D})}}}}(V)=0$ if $|V|$ is central with $|V|>2$, and $\delta_{\ol{H},H_{\ol{\Omega_{(\mathcal{D})}}}}(V)=1$ if $|V|=2$ and $2$ divides $|\PDelta_1/\POmega_1|$. (Since all modules for $\Sym_2$ over $\mathbb{F}_2$ are trivial.) 
  
Next, note that $\delta_{H_{\overline{\Omega}/H_{\overline{\Omega(\mathcal{D})}}}}(V)\le 1$, with equality if and only if either ($|V|=2$ with $t\neq 3$ mod $4$, $t\neq 4$, and the case $t=2$ with $m=2$ mod $4$ does not hold) or ($V\cong \Alt_t$ and $t\neq 2,4$) or ($t=4$, and $|V|=2$ or $|V|=3$), using the argument immediately following (\ref{7Bound}) above. Here, we are also using Lemma \ref{KeyLemma} Part (iv). In particular, note that $\delta_{H_{\overline{\Omega}/H_{\overline{\Omega(\mathcal{D})}}}}(V)=0$ if $V$ is abelian with $|V|>2$ (in particular, if $V$ is non-central).

Finally, we consider the group $\overline{H}/H_{\overline{\Omega}}\le \Out(\ol{\Omega})$. Since this group is soluble with at most one abelian non-central chief factor, the bound $\delta_{\overline{H}}(V)\le 3$ (or $2$ in the case $V$ is non-central) now follows immediately from (\ref{7Bound}) in all cases, except when $|V|=2$ and either $t=2$ and $m\neq 2$ (mod $4$), or $m\neq 3$ (mod $4$). Also, arguing as in Proposition \ref{C4}, one of the following must occur.\begin{enumerate}[(1)]
\item Case $\Lf$ with $f$ even, $q$ odd, and $\delta_{\overline{H}/H_{\overline{\Omega}}}(V)\geq 2$. Here, we have $\delta_{\overline{H}}(V)= \delta_{ H_{\overline{\Omega(\mathcal{D})}}}(V)+ \delta_{H_{\overline{\Omega}/H_{\overline{\Omega(\mathcal{D})}}}}(V)+\delta_{\overline{H}/H_{\overline{\Omega}}}(V)=0+1+\delta_{\overline{H}/H_{\overline{\Omega}}}(V)$.
\item $\PDelta_1=\PGO^{\pm}(q)$ with $q$ odd, $H_{\overline{\Delta(\mathcal{D})}}=(\POmega_1.C)^t$, where $C\in\{2^2,D_8\}$. Here, we are in case $\Of^{\epsilon}$. As above we have, in this case, $\delta_{\overline{H}}(V)= \delta_{ H_{\overline{\Delta(\mathcal{D})}}}(V)+ \delta_{H_{\overline{\Delta}/H_{\overline{\Delta(\mathcal{D})}}}}(V)+\delta_{\overline{H}/H_{\overline{\Delta}}}(V)=0+1+\delta_{\overline{H}/H_{\overline{\Delta}}}(V)$.\end{enumerate}
Thus, we get $\delta_{\ol{H}}\le 4$, with equality if and only if $|V|=2$ and one of the cases from the $\mathcal{C}_7$ cases in Table 1 occurs.
The proof is complete.\end{proof}
 
\begin{prop}\label{C8} Suppose that $H$ lies in class $\mathcal{C}_8$. Then Theorem \ref{CrownLanguage} holds.\end{prop}
\begin{proof} The argument here is almost identical to the argument used to prove Proposition \ref{C5}: we repeat the details here for the reader's benefit. The group $H_{\Gamma}$ here is a classical group itself, of dimension $n$ over $\mathbb{F}$: we write $\kappa_1$ for the associated bilinear or quadratic form. See \cite[Table 4.8.A]{KL} for the precise possibilities for $H_{\Gamma}$. 

Let $I_1$ be the group of isometries of $(\Vt,\kappa_1)$, and similarly define $\Omega_1$, and the projective groups $\PI_1$ and $\POmega_1$. The group $\POmega_1$ is normal in $\ol{H}$ by \cite[Proof of Proposition 4.8.2]{KL}. Furthermore, $\Omega_1$ is absolutely irreducible by \cite[Proposition 2.0.9]{KL}. Hence, $C_{H}(\Omega_1)=\mathbb{F}^{\ast}$. It follows that $C_{\ol{H}}(\POmega_1)$ is trivial, and hence that $\ol{H}/\POmega_1\le \Out(\POmega_1)$. If $\POmega_1\cong \Sf_4(2)$, then $\ol{H}\le \Aut(\Alt_6)=\Alt_6.2^2$, and $\delta_{\ol{H}}(V)\le 2$. If $\POmega_1$ is a non-abelian simple group then $\ol{H}$ is almost simple, and the required bounds $\delta_{\ol{H}}(V)\le 1$ [respectively $2$, $3$] if $V$ is non-abelian [resp. non-central, central] follow from Proposition \ref{p:alsimple}. Otherwise the possibilities for $\ol{H}$ are listed in \cite[Proposition 2.9.2]{KL}, and the required bounds follow easily by direct computation.\end{proof}

\section{Almost simple groups with exceptional socle}\label{ExSpor}
In this section we deal with the cases where the socle $G_0$ of $G$ is an exceptional simple group. Thus, there exists a simple algebraic group $\mathrm{G}$ of adjoint type such that $G_0=(\mathrm{G}_{\sigma})'$ is the derived group of the fixed point subgroup of $\mathrm{G}$ under a Steinberg endomorphism $\sigma$ of $\mathrm{G}$. Throughout, we will write $H_0:=H\cap G_0$. Also, for a subgroup $X$ of $\mathrm{G}$, we will write $X_{\sigma}$ for the group $X\cap \mathrm{G}_{\sigma}$. Finally, we will write $I=\Inndiag(G_0)$ for the group of inner diagonal automorphisms of $G_0$. Note that $I=\mathrm{G}_\sigma$.

By \cite{LiebeckSeitz}, the possibilities for $H$ are as follows.\begin{enumerate}[(i)]
\item $H$ is a maximal parabolic subgroup of $G$.
\item $H$ is almost simple.
\item $H=N_G(D_{\sigma})$ is the normaliser of a connected reductive subgroup $D$ of $\mathrm{G}$ of maximal rank.
\item $H=N_G(E)$, where $E$ is an elementary abelian $2$-group. 
\item $F^{\ast}(H)$ is as in \cite[Table III]{LiebeckSeitz}.\end{enumerate} 
If $H$ is as in Case (ii), then the bound $d(H)\le 3$ follows immediately from Proposition \ref{p:alsimple}. Thus, we need only address Cases (i), (iii), (iv) and (v).

We begin with Case (i).
\begin{prop}\label{ExParCase} Let $G$ be an almost simple group with exceptional socle $G_0$ and let $H$ be a maximal parabolic subgroup of $G$. Then Theorem \ref{CrownLanguage} holds. 
\end{prop}
\begin{proof} Let $\Phi:=\{\pm\phi_1,\hdots,\pm\phi_r\}$ be a set of fundamental roots for $G$, and denote by $U_{\alpha_i}$ the root subgroups. Also, for $I\subseteq\{1,\hdots,r\}$, set $U_I:=\prod_{j\not\in I}U_{\alpha_j}=\langle U_{\alpha_j}\text{ : }j\not\in I\rangle$. Since $H$ is maximal, $H$ is one of the following types:\begin{enumerate}[(a)]
\item $H_0=N_{G_0}(U_{\alpha_i})$ is the parabolic subgroup obtained by deleting node $i$ from the Dynkin diagram for $G$; or
\item $G_0$ is of type $E_6$; $F_4$ with $p=2$;, $B_2$ with $p=2$; or $G_2$ with $p=3$; and $H_0=N_{G_0}U_{\{i,j\}}$ is the parabolic subgroup obtained by deleting nodes $i$ and $j$ from the Dynkin diagram for $G$.\end{enumerate}
Furthermore, as in the proof of Proposition \ref{C1Parabolic}, we have $H_0=C\rtimes L$, where $C$ is the unipotent radical of $H_0$ and $L$ is a Levi subgroup. Since $C$ is characteristic in $H_0\unlhd H$, both $C$ and $C'$ are normal subgroups of $H$. Thus, we have a normal series
\begin{align}\label{EPSeries} 1\le C'\le C \le C\rtimes L=H_{0}\le {H}.\end{align}
Our strategy will be to investigate the structure of the factors in this series, and then apply Lemma \ref{ElementaryLemma}.  

We first consider the structure of $C$. Assume that we are in Case (a), so that $H_0=P_i$. Here, the group $C/C'$ is an irreducible module for $L$ over the field $\mathbb{F}_q$ of definition for $G_0$, by \cite[Theorem 2(a)]{ABS}. Furthermore, since $C$ is the unipotent radical of $H_0$, it is nilpotent. Hence the group $C/\Frat(C)$ is elementary abelian. Hence, $C'\le \Frat(C)$. It now follows from the above arguments that $C/C'$ is an irreducible $\mathbb{F}_q[H]$-module, and that $C'\le \Frat(C)\le \Frat(H)$.

Assume now that we are in Case (b). Then it follows from \cite{ABS} that $C/C'$ is either irreducible as an $L$-module, or has two non-isomorphic irreducible composition factors. In the latter case, we have a series $1\le C'<C_1<C$ for $C$, where each term is normalised by $L$. Moreover, $C_1/C'$ and $C/C_1$ are non-isomorphic as $L$-modules, and hence non-isomorphic as $H_0=C\rtimes L$ modules. Since $C/C'$ is an $H$-module and $H_0\unlhd H$, it follows from Clifford's Theorem that either $1\le C'<C_1<C$ is also an $H$-series (in which case $C_1/C'$ and $C/C_1$ are irreducible $H$-modules), or $C/C'$ is irreducible.  Either way, we conclude that 
\begin{align}\label{EPSeries1} 1\le C'\le C_1\le C \le C\rtimes L=H_{0}\le {H}\end{align}
is a normal series for $H$, with $C'\le \Frat(H)$, $C_1/C'$ a chief factor of $H$, and $C/C_1$ either trivial or a chief factor of $H$ which is not $H$-equivalent to $C_1/C'$.

We now consider the structure of $L$. We have $L=T\prod_{j\neq i}U_{\pm\alpha_j}$, where $T$ is a maximal torus in $G_0$. Now $L_j:=U_{\alpha_j}U_{\alpha_{-j}}$ is either a quasisimple group of Lie type of rank $1$, or $U_{\alpha_j}U_{\alpha_{-j}}$ is isomorphic to $L_2(2)$ or $L_2(3)$ (see \cite[Corollary 25.3]{Humphreys}). Furthermore, $U_{\alpha_j}$ centralises $U_{\pm\alpha_k}$ for $j\neq k$, and each $L_j$ contains a subgroup of $T$ of codimension $1$. It follows that $M:=\prod_{j\neq i}L_j\unlhd L$ is a central product, and $L/M$ is cyclic. Finally, $L_j\not\cong L_k$ for $j\neq k$. This can be quickly deduced from examining the Dynkin diagrams of the simple exceptional groups.

From the arguments in the previous two paragraphs, we deduce that   
\begin{align}\label{EPSeries2} 1\le C'\le C_1\le C \le C\rtimes M\le C\rtimes L=H_{0}\le {H}\end{align}
is a normal series for $H$, with $C'\le \Frat(H)$; $C_1/C'$ a chief factor of $H$; $C/C_1$ either trivial or a chief factor of $H$ which is not $H$-equivalent to $C_1/C'$; $M$ a central product of non-isomorphic groups $L_j$, where each $L_j$ is either quasisimple or $L_j\in\{\SL_2(2),\SL_2(3)\}$. Thus, by Lemma \ref{ElementaryLemma} we have
\begin{align}\label{EPSeriesBound} \delta_{{H}}(A)=\delta_{{H},{C}/C'}(A)+\delta_{{H},{M}}(A)+\delta_{{H},L/{M}}(A)+\delta_{{H},{H}/H_{0}}(A).\end{align} 
As in the proof of Proposition \ref{C1Parabolic}, we now proceed to bound each of the quantities on the right hand side of (\ref{EPSeriesBound}).

First, the dimensions over $\mathbb{F}_p$ of the composition factors of the $\mathbb{F}_p[{H}]$-module $C/C'$ are greater than $2$ by \cite[Theorem 2]{ABS}. In particular, $\delta_{{H},C/C'}(A)\le 1$, and $\delta_{{H},C}(A)=0$ if either $A$ is not a $p$-group, or $A$ is a $p$-group with $\dim_{\mathbb{F}_p}(V)\le 2$.  

Next we have, by Lemma \ref{QuasisimpleLemma}, that $\delta_{{H},M}(A)\le 1$, and $\delta_{{H},M}(A)=0$ if $A$ is abelian and $A$ is not isomorphic to $2$, $3$, or $2^2$ (these come from the non-Frattini chief factors of $\SL_2(2)$ and $\SL_2(3)$). Also, $L/M$ is cyclic, so $\delta_{{H},L/M}(A)\le 1$. Finally, note that $H/H_0$ has at most two non-Frattini chief factors, each cyclic, using Table \ref{ta:outgroup}. 

If $A$ is non-abelian, then $\delta_H(A)\le 1$, using the above arguments and the fact that $H/H_0$ is soluble. So assume that $A$ is abelian. If $A$ is non-central, then $\delta_{{H},{H}/H_{0}}(A)\le 1$ by the previous paragraph, so $\delta_{{H}}(A)\le 2$ overall (this is because any non-central non-Frattini chief factor of $H/H_0$ has dimension $1$). So all that remains is to prove that $\delta_{{H}}(A)\le 3$ when $A$ is central. For this to fail, we would need $|A|=2$ or $|A|=3$, $\SL_2(2)$ or $\SL_2(3)$ would have to occur in $M$ [respectively], $L/M$ would have to be divisible by $2$ or $3$ [resp.], and centralised by all outer automorphisms of $G_0$ in $H$, and $H/H_0$ would have to have a factor group isomorphic to $2^2$ or $3^2$ [resp.]. By examining Table \ref{ta:outgroup}, and the orders of the maximal tori in the exceptional simple groups (for example, see \cite[Table 4]{BLS}), we see that this can never occur, and so the proof is complete.\end{proof}

Next, we consider Case (iii).
\begin{prop}\label{ExMaxRankCase} Let $G$ be an almost simple group with exceptional socle $G_0$ and let $H$ be a maximal subgroup of $G$ as in Case (iii) above, where $D$ is not a torus. Then Theorem \ref{CrownLanguage} holds. 
\end{prop}
\begin{proof} Since $D$ is connected and reductive, we have $D=MS$, where $M:=D'$ is semisimple and $S$ is a torus. Furthermore, $N_I(D)=D_{\sigma}.W_{\Delta_D}$, where $W_{\Delta_D}$ is a certain subquotient of the Weyl group of $G_0$, by \cite[Lemma 1.2]{LSS}. The structure of $N:=N_I(D)$ is given in \cite[Tables 5.1 and 5.2]{LSS}. Note that $N_{I'}(D)\le H_I:=H\cap I\le N$, and $N/N_{I'}(D)$ is cyclic of order dividing $I/I'$, the group of diagonal automorphisms of $G_{0}$. Thus,
\begin{align}\label{OutObs} H\cap N\unlhd N \text{ and }N/H\cap N\text{ has order dividing }I/I'.\end{align}

We now proceed to examine each of the cases in \cite[Table 5.1]{LSS}. Recall that we are trying to prove that
\begin{align}\label{ExBound} \delta_H(A)\le 3\text{ if $A$ is abelian, and }\delta_H(A)\le 2\text{ if $A$ is non-abelian.}\end{align}
Using the bound $\delta_{H}(A) \le \delta_{H,H_I}(A)+\delta_{H/H_I}(A)$, and appealing to Table 2 for the chief factors of $H/H_I$, the result is clear in most cases. For example, consider $G_0=^3D_4(q)$. One of the two possibilities for $H_I$ in Table 5.1 has shape $K.a$, where $a$ divides $(2,q-1)$, and $K$ is a central product of two quasisimple groups. Hence, $K$ has shape $Z.(L_2(q)\circ L_2(q^3))$, where $Z$ is Frattini in $K$ (and hence Frattini in $H$, since $K$ is subnormal in $H$). We then get $\delta_H(A)\le 2$ for any $A$, since $H/H_I$ is cyclic (see Table 2). This approach works in almost all cases, except for some of those maximal subgroups in the case $G_0=E^{\epsilon}_l(q)$, with $l\in\{6,7,8\}$. These are dealt with in the paragraphs below.

Suppose first that $G_0=E^{\epsilon}_6(q)$, so that $H/H_I$ is a subgroup of $Z_f\times Z_2$ (see Table 2), and that $N=E.\Sym_3$, where $E=J\circ R$, with $J\cong\Omega^{+}_8(q)$ and $R=(q-\epsilon)^2$.

Let $T$ be a $\sigma$-stable maximal torus of $G$ contained in $D$. Let $\Phi_D$, $\Phi$ denote the root systems of $D$, $G$ with respect to $T$, respectively. In particular, $\Phi_D\subset \Phi$ is a subset of the group of rational characters $X:=X(T)\cong \mathbb{Z}^6$ of $T$. Recall that the integer $q$ here is a power of $p$ defined by the action of $\sigma$ on $X$, which is given by $\sigma=q\sigma'$, where $\sigma'$ is a permutation of $\Phi$ which reverses the Dynkin diagram of $G$ in case $\epsilon=-1$, and $\sigma':=1$ in case $\epsilon=+1$. More precisely, if $\chi\in X$, then $\chi^{\sigma}$ is defined by $\chi^{\sigma}(t):=\chi^{\sigma'}(t^q)$, for $t\in T$.

Let $W_D$, $W$ be the Weyl groups of $D$, $G$ respectively, so that $W_D\le W$. Fix fundamental systems $\Delta_D\subset \Phi_D$ and $\Delta\subset \Phi$ for $D$ and $G$ respectively. Without loss of generality, we will take $\Phi$ to be as in \cite[Page 260]{Bourbaki}, with set of fundamental roots $\Delta:=\{\alpha_1,\hdots,\alpha_6\}$. The longest root $\alpha_0$ is defined by $\alpha_0:=\alpha_1+2\alpha_2+2\alpha_3+3\alpha_4+2\alpha_5+\alpha_6$. From this, we may form the extended Dynkin diagram of type $E_6$ with vertices $\widetilde{\Delta}:=\{-\alpha_0,\alpha_1,\hdots,\alpha_6\}$.

Now, let $\Phi_D$ be the $D_4$-subsystem of $E_6$ with fundamental roots $\Delta_D=\{\alpha_2,\alpha_3,\alpha_4,\alpha_5\}$, so that $V$ has basis $\{\alpha_1,\alpha_6\}$. Let $W_{\Delta_D}:=N_W(W_{D})/W_{D}$. Using MAGMA, one can see that $N_W(\Delta_D)\cap N_W(\wt{\Delta})$ contains an element $a$ whose induced action on $\wt{\Delta}$ is $(\alpha_2,\alpha_3,\alpha_5)(\alpha_1,\alpha_6,-\alpha_0)$. Similarly, $N_W(\Delta_D)$ contains an element $b$ which acts on $\Delta_D$ as $(\alpha_3,\alpha_5)$ (but this element does not stabilise $\wt{\Delta}$). Furthermore, $\alpha_1^b=\alpha_6$ modulo $X_D\otimes \mathbb{Q}$. Since $\alpha_0=\alpha_1+\alpha_6$ modulo $X_D\otimes \mathbb{Q}$, it follows that the induced action of $W_{\Delta_D}$ on $V$ is
\begin{align}\label{Mark}W_{\Delta_D}^V=\left\langle \begin{bmatrix} -1 & -1\\1 & 0\end{bmatrix},\begin{bmatrix} 0 & 1\\1 & 0\end{bmatrix}\right\rangle\cong {\Sym}_3. \end{align}
Now, recall that the action of $\sigma$ on $X(T)$ commutes with the action of $W$ on $X(T)$. By \cite[Proposition 7 and the discussion following]{CarterC}, $V/(\sigma-1)V$ is isomorphic as a $W_{\Delta_D}$-group to $S_{\sigma}$. Hence, since $S_{\sigma}\cong (q-\epsilon)^2$, we may write
$$V/(\sigma-1)V=\langle \chi\rangle\times \langle \lambda\rangle\cong (q-\epsilon)^2.$$
We have $\chi^a=\chi^{-1}\lambda^{-1}$ and $\chi^b=\lambda$ by (\ref{Mark}). In particular, it follows that for any prime $p$ dividing $q-1$, $W_{\Delta_D}=\langle a,b\rangle\cong \Sym_3$ acts irreducibly on 
$$Y_p:=[S_{\sigma}/O^{p}(S_{\sigma})]/[\Frat (S_{\sigma}/O^{p}(S_{\sigma}))]\cong p^2.$$
Furthermore, $Y_2.\langle a\rangle \cong \Alt_4$ and $Y_2.\langle b\rangle\cong D_8$.
    
With the information deduced above, we can now determine the required upper bounds on $\delta_H(A)$, for each non-Frattini chief factor of $H$. We have a normal series $$1 < M_{\sigma}<H\cap D_{\sigma} < H\cap N < H \le (H\cap N).\langle\phi\rangle.\langle g\rangle,$$
where $\phi$ represents a field automorphism of $G_0$, and $g$ is a graph automorphism. The field and graph automorphisms normalise $D_{\sigma}$, by \cite[Lemma 1.3]{LSS}.

We first consider the $H$-chief factors contained in $M_{\sigma}=D_{\sigma}'$. Let $Z=Z(D_{\sigma})=M_{\sigma}\cap S_{\sigma}$. By the usual arguments, $Z$ is contained in the Frattini subgroup of $D_{\sigma}'$, and hence of $H$. We have $D_{\sigma}/M_{\sigma}\cong S_{\sigma}/Z$, where $M$ and $S$ are as above. Also, $M_{\sigma}\cong \Omega_8^{+}(q)$, and $S_{\sigma}\cong (q-1)^2$. Note that $D_{\sigma}$, and hence $H$, also has a unique non-abelian chief factor, isomorphic to $M_{\sigma}/Z$.

We now consider the $H$-chief factors contained in $H\cap D_{\sigma}/M_{\sigma}\le S_{\sigma}/Z$. Indeed, by the arguments in the paragraph above, they can be completely determined by the group $H\cap N/H\cap D_{\sigma}$: Since $N/D_{\sigma}\cong W_{\Delta_D}\cong \Sym_3$ (see \cite{LSS}), and $|N/H\cap N|$ divides $3$ by (\ref{OutObs}), this group has order either $2$ or $6$. If it has order $2$, then $H\cap N$ acts on $S_{\sigma}$ by switching the cyclic factors, so we get at most two chief factors of order $p$ for each prime $p$ dividing $q-\epsilon$, except that one of these is Frattini in the case $p=2$. Otherwise, we get at most $1$ non-Frattini $H$-chief factor in $S_{\sigma}/Z$ for each prime $p$ dividing $q-\epsilon$. All of this information follows from the arguments above.  

Since $\sigma'$ and $\phi$ commute, it now follows that $\delta_H(A)\le 3$ for all $A$, with equality possible only if $|A|=2$. This gives us what we need.

Next, assume that $G_0=E_7(q)$, so that $H/H_I$ is cyclic of order dividing $f$ (see Table 2), and that $N=(J_1\circ J_2).d^3.\Sym_3$, where $J_1=\SL_2(q)^3$, $J_2=\Omega^+_8(q)$ and $d:=(2,q-1)$. In order to determine the chief factors of $H$, we will need to examine the extensions within $H_I$ more carefully. The extended Dynkin diagram of $E_7$ has a subdiagram $3A_1+D_4$, which is obtained by deleting the 6th node of the Dynkin Diagram of $E_7$ (see \cite[Page 262]{Bourbaki}).

Thus, $D$ is a semisimple algebraic group with fundamental root system $$\Psi:=\{-\alpha_0,\alpha_1,\alpha_2,\alpha_3,\alpha_4,\alpha_5,\alpha_7\},$$ where $\alpha_0:=2\alpha_1+2\alpha_2+3\alpha_3+4\alpha_4+3\alpha_5+2\alpha_6+\alpha_7$. Consider the associated simply connected group $D_{sc}=\SL_2(K)^3\times \Spin_8(K)$, and the natural isogeny $\pi_1:D_{sc}\rightarrow D$. Set $Z:=\Ker(\pi_1)\le Z(D_{sc})$. By \cite[Lemma 24.20]{MT}, we have a long exact sequence 
$$1\rightarrow Z_{\sigma}\rightarrow (D_{sc})_{\sigma}\rightarrow D_{\sigma}\rightarrow L(D_{sc})\cap Z/L(Z),$$
where $L:D_{sc}\rightarrow D_{sc}$ is defined by $L(x):=(x{\sigma})x^{-1}$. Now $Z(D_{sc})\cong 2^5$, and clearly $\sigma$ fixes $Z(D_{sc})$ (in this case, $\sigma$ is a Frobenius automorphism, $x\rightarrow x^q$). Also, $D$ is connected so $L$ is surjective by the Lang-Steinberg Theorem (see \cite[Theorem 21.7]{MT}). By \cite[Proposition 9.15]{MT}, $Z=Z_{\sigma}\cong (\Lambda(D))_{p'}$, where $\Lambda(D)$ denotes the fundamental group of $D$. Since $\mathrm{G}$ is of adjoint type, we have $X(T)=\mathbb{Z}\Phi$. Hence, $[X(T):\mathbb{Z}\Psi]=[\mathbb{Z}\Phi:\mathbb{Z}\Psi]=2$. Since $T$ is also a maximal torus in $D$, and the full fundamental group of $D$ is elementary abelian of order $2^5$, it follows that $Z_{\sigma}\cong 2^4$. Since $(D_{sc})_{\sigma}\cong \SL_3(q)^3\times \Spin^{+}_8(q)$, and $L(D_{sc})\cap Z/L(Z)\cong 2^4$ by the above arguments, we have that $D_{\sigma}\cong 2.(L_2(q)^3\times \POmega^+_8(q)).2^4$. Furthermore, if we write $Z(D_{sc})=\langle z_1,z_2,z_3,z_4,z_5\rangle\cong 2^5$, we can see from the Lang-Steinberg Theorem that there exists $y_i\in D_{sc}$ such that $(y_i\sigma) y_i^{-1}=z_i$: set $a_i:=y_i\pi_1\in D$. Then since $\pi_1\sigma=\sigma\pi_1$, we have $a_i{\sigma}=(y_iz_i){\pi_1}=y_i\pi_1=a_i$, so $a_i\in D_{\sigma}$. From the definition of the $a_i$ it is easy to see that the group $B$ generated by their images generate $D_{\sigma}/((D_{sc})_{\sigma})\pi$. Furthermore, if $b_i$ is a generator of the cyclic group of diagonal automorphisms of $X_i$ for $1\le i\le 3$, and $c_1$, $c_2$ generates the group of diagonal automorphisms of $\POmega^+_8(q)$, then it is easy to see that $B$ is isomorphic to the subgroup $\langle b_1b_2,b_2b_3,b_1c_1,b_1c_2\rangle\le \langle b_1\rangle\times \langle b_2\rangle\times \langle b_3\rangle\times\langle c_1,c_2\rangle$.

Finally, by using MAGMA for example, one can see that $W_{\Delta_D}\cong \Sym_3$ acts regularly on the non-zero elements of both $B_1:=B\cap (\langle b_1\rangle\times \langle b_2\rangle\times \langle b_3\rangle)$ and $B/B_1$ (which are both Klein $4$-groups). Furthermore, $W_{\Delta_D}$ acts faithfully on the three factors of $L_2(q)$ in $D_{\sigma}$. Combining all of this information, we see that $H$ has a Frattini chief factor of order $2$ (namely $Z(D_{\sigma})$); at most $2$ non-Frattini chief factors of order $2$ (coming from $W_{\Delta_D}$ and the cyclic group of field automorphisms); at most $2$ chief factors of order $2^2$ ($B_1$ and $B/B_1$); and $1$ chief factor of order $p_1$ for each prime $p_1$ dividing $|H/H_I|$ (which is cyclic of order at most $f$). This gives us what we need. 

The approach above can also be used to prove that (\ref{ExBound}) holds in the following cases:\begin{itemize}
\item $G_0=E^{\epsilon}_6(q)$ and $N=J.\Sym_3$, where $J=e.(L_3(q)\times L_3(q)\times L_3(q)).e^2$, and $e:=(3,q+\epsilon)$.
\item $G_0=E_7(q)$ and $N=J.L_3(2)$, where $J=d^3.L_2(q)^7.d^4$, and $d:=(2,q-1)$.
\item $G_0=E_8(q)$ and $N=J.(\Sym_3\times 2)$, where $J=d^2.\POmega^+_8(q)^2.d^2$, and $d:=(2,q-1)$.
\item $G_0=E_8(q)$ and $N=J.\GL_2(3)$, where $J=e^2.L^{\epsilon}_3(q)^4.e^2$, and $e:=(3,q-\epsilon)$.
\item $G_0=E_8(q)$ and $N=J.\AGL_3(2)$, where $J=d^4.L^{\epsilon}_2(q)^8.d^4$, and $d:=(2,q-1)$.\end{itemize}
Using the same notation as above, $W_{\Delta_D}$ is $\Sym_3$, $L_3(2)$, $\Sym_3\times 2$, $\GL_2(3)$, and $\AGL_3(2)$ in these cases, respectively. The proof is complete.
\end{proof}

\begin{prop}\label{ExTorusCase} Let $G$ be an almost simple group with exceptional socle $G_0$ and let $H$ be a maximal subgroup of $G$ as in Case (iii) above, where $D$ is a torus. Then Theorem \ref{CrownLanguage} holds. 
\end{prop}
\begin{proof} Here we have $H=N_G(T_{\sigma})$, where $T$ is a $\sigma$-stable maximal torus in $\mathrm{G}$. Furthermore, the groups $T_{\sigma}$ and $N_{\mathrm{G}_{\sigma}}(T_{\sigma})/T_{\sigma}$ are given in \cite[Table 5.2]{LSS}. The latter is in fact a subgroup $C$ of the Weyl group $W$ of $\mathrm{G}$ (see \cite[II, (1.8)]{Seitz}), and our strategy will be to determine the action of $C$ on the abelian group $T_{\sigma}$: from this we can completely determine the non-Frattini chief factors of $N_{\mathrm{G}_{\sigma}}(T_{\sigma})$. We will frequently use the fact that 
\begin{align}\label{IrrPoint}H\cap I\geq N_{G_0}(T_{\sigma})=N_{\mathrm{G}_{\sigma}}(T_{\sigma})T_{\sigma}\end{align}
to deduce facts about the non-Frattini $H$-chief factors contained in $H\cap I$. 

The abelian group $T_{\sigma}$ is the direct product of its $l$-parts (for $l$ prime), so the non-Frattini chief factors will be determined by the $\mathbb{F}_l[C]$-modules $M_l:=(T_{\sigma})_l/\Frat((T_{\sigma})_l)$. If $l$ does not divide $|C|$, then $M_l$ lifts to a $\overline{\mathbb{Q}}[C]$-module $\overline{M_l}$, and the dimensions of the composition factors of ${M_l}$ are the same as those for $\overline{M_l}$. In the finitely many cases where $l$ divides $|C|$, we can use MAGMA to find the composition factors of $M_l$ as a ${\mathbb{F}_l}[C]$-module. Apart from two cases, we find that $M_l$ is irreducible as an ${\mathbb{F}_l}[C]$-module, and hence as an ${\mathbb{F}_l}[H]$-module, by (\ref{IrrPoint}). In most cases $\dim{M_l}$ is large, so one can immediately deduce that $\delta_{H}(M_l)=1$ for all $l$. In general, we can quickly deduce from Table 2 and \cite[Table 5.2]{LSS} that $\delta_{H}(M_l)\le 2$ if $|M_l|>2$, and $\delta_{H}(M_l)\le 3$ if $|M_l|=2$. The remaining chief factors of $H$ are all contained in $C$ and $H/H_0$, and we can deduce the bound for $\delta_H(A)$ immediately from Table 2 and \cite[Table 5.2]{LSS} again.

This leaves us with the cases mentioned above where $M_l$ is reducible an ${\mathbb{F}_l}[C]$-module. By the arguments above, we may assume that $M_l$ is reducible as an $H$-module as well. These cases are as follows:\begin{enumerate}[(i)]
\item $G=E_6^{\epsilon}(q)$, $3$ divides $q-\epsilon$, $T_{\sigma}=(q-\epsilon)^6$, and $l=3$. In this case $M_l$ has two $C$-composition factors $A_l$ and $B_l$, of dimensions $1$ and $5$ respectively. Since $M_l$ is also reducible as an $H$-module, we deduce from (\ref{IrrPoint}) that $A_l$ and $B_l$ are also $H$-composition factors. Since $C=W(E_6)=S.2$ where $S$ is non-abelian simple, and $H/H_0$ has at most two non-Frattini chief factors of order $3$, at least one of which must be central, the required bound follows for $A=B_l$. Clearly $\delta_H(B_l)=1$, and so we have what we need.
\item $G=E_7(q)$, $q$ is odd, $T_{\sigma}=(q-\epsilon)^6$ and $l=2$. Here $M_l$ has two $C$-composition factors of dimensions $1$ and $6$, and the argument follows as above.
\end{enumerate}
\end{proof}

\begin{prop}\label{Ex45Case} Let $G$ be an almost simple group with exceptional socle $G_0$ and let $H$ be a maximal subgroup of $G$ as in Case (iv) or (v) above. Then Theorem \ref{CrownLanguage} holds. 
\end{prop}
\begin{proof} The structure of the group $N_{\mathrm{G}_{\sigma}}(H)$ is given in \cite[Table III]{LiebeckSeitz} in case (iv), and the bounds on $\delta_H$ follow easily in every case. 

So assume that we are in case (iv). Then $H$ normalises an elementary abelian subgroup $E$ of $\mathrm{G}_{\sigma}$, and the structure of $N_{\mathrm{G}_{\sigma}}(E)$ is given by \cite[Theorem 1 Part (II)]{CLSS}. In particular, $N_{\mathrm{G}_{\sigma}}(E)/C_{\mathrm{G}_{\sigma}}(E)$ acts irreducibly on $E$, so $\delta_{H,E}(E)=1$. 

Suppose first that $\mathrm{G}=E_8$ (so that $I\le H$), and $H\cap I=N_{\mathrm{G}_{\sigma}}(E)=2^{5+10}.SL_5(2)$. In this case, $E=2^5$, $N_{\mathrm{G}_{\sigma}}(E)/C_{\mathrm{G}_{\sigma}}(E)$ acts naturally on $E$, and $C_{\mathrm{G}_{\sigma}}(E)$ is special of order $2^{15}$. Since $\SL_5(2)$ is simple and $I/H\cap I$ is cyclic, we must have that $H\cap I=N_{\mathrm{G}_{\sigma}}(E)=2^{5+10}.\SL_5(2)$. In particular, $Z(C_{\mathrm{G}_{\sigma}}(E))=E$, from which it follows that $\delta_{H,C_{\mathrm{G}_{\sigma}}(E)}(E)=1$. From the proof of \cite[Theorem 1 Part (II)]{CLSS}, we can see that $C_{\mathrm{G}_{\sigma}}(E)/E=U\oplus U'\cong 2^5\oplus 2^5$ as an $\SL_5(2)$-module, with the $\SL_5(2)$, acting naturally on each $U$, $U'$. Clearly we have $\delta_H(A)\le 2$ for $A\in \{E,U,U'\}$. Finally, since $\Out(G_0)$ is cyclic, we have $\delta_{H}(A)\le 1$ for any other chief factor $A$ of $H$. 

The same strategy as above works in the case $\mathrm{G}=E_6^{\pm}$ and $N_{\mathrm{G}_{\sigma}}(E)=3^{3+3}.SL_3(3)$, except that in this case, we have $H\cap I=N_{\mathrm{G}_{\sigma}}(E)$ since $N_{\mathrm{G}_{\sigma}}(E)$ is perfect and $N_{\mathrm{G}_{\sigma}}(E)/H\cap I$ is cyclic.

Next, assume that $\mathrm{G}=E_7$ (so that we again have $I\le H$), and $H\cap I=N_{\mathrm{G}_{\sigma}}(E)=(2^2\times \Inndiag(\POmega_8^+(q))).\Sym_3$. Here, the $\Sym_3$ on top acts naturally as $\GL_2(2)$ on $E=2^2$, and $H\cap I/E=\Aut(\POmega_8^+(q))$. Hence $\delta_H(E)=2$. Furthermore, since $\Out(G_0)=Z_{(2,q-1)}\times Z_f$, we have $\delta_H(Z_2)\le 3$, $\delta_H(\POmega_8^+(q))=1$, and $\delta_H(Z_3)\le 2$ if the action on $Z_3$ is trivial, and $\delta_H(Z_3)= 1$ otherwise.

The remaining cases are easier and follow from the same arguments as above - in every case we have $N_{\mathrm{G}_{\sigma}}(E)=p^{n}.S$, with $S\le \GL_n(p)$ simple and irreducible. We then use Table 2 to bound the contribution of the outer automorphism group to $\delta_H$ in each case, and this completes the proof.\end{proof}           
We conclude the paper by showing that there exists infinitely many pairs $(G,H)$ where $G$ is almost simple, $H$ is a maximal subgroup of $G$, $d(H)=3$, and $H$ dos not have an elementary abelian quotient of order $l^3$, for any prime $l$. Thus, our theorem is ``best possible".
\begin{ex}\label{iEx} Let $p$ be a prime such that $p\equiv 1$ (mod $3$), let $n=2m$, with $m$ odd and $3$ dividing $m$, and set $\ol{G}:=\Aut(\Lf_n(p))$. Let $H$ be a non-parabolic maximal subgroup of $G$ in class $\mathcal{C}_1$. Adopting the same notation as used in the proof of Proposition \ref{C1}, we have $\ol{H}/\ol{L}\cong (Z_{d_1}\times Z_{d}).2$, where $d:=(n,p-1)$, and $d_1:=(m,p-1)$. The $Z_d$ here represents the group of diagonal automorphisms of $\Lf_n(p)$. In particular, $\ol{H}$ has a factor group $\ol{H}/R\cong 3^2:2$, where the $2$ acts by inverting the non-zero elements in the $3^2$. Hence, $d(\ol{H}/R)=3$. We have $d(\ol{H})\le 3$ by the main theorem. Since $\ol{L}$ is a central product of two quasisimple groups and $d_1$ is odd, we have $\delta_{\ol{H}}(Z_2)\le 2$, for any other chief factor $A$ of $\ol{H}$.\end{ex}

\end{document}